\documentclass[a4paper,12pt]{article}
\usepackage[utf8]{inputenc}
\usepackage{fullpage}
\usepackage{amsmath, amssymb, amstext, amsfonts, amsthm, bbm, array, enumerate,scalerel}
\usepackage{mathrsfs, dsfont,mathabx}
\usepackage{nicefrac}
\usepackage[none]{hyphenat}
\usepackage{xcolor}
\usepackage{tikz}
\usepackage{stmaryrd}
\usepackage{tikz-3dplot}
\usetikzlibrary{patterns}
\usetikzlibrary{plotmarks}

\usepackage{bbm}

\usepackage{caption}
\usepackage{subcaption}

\newtheorem{theorem}{Theorem}[section]
\newtheorem{lemma}[theorem]{Lemma}
\newtheorem{corollary}[theorem]{Corollary}
\newtheorem{proposition}[theorem]{Proposition}

\theoremstyle{definition}
\newtheorem{definition}[theorem]{Definition}
\newtheorem{example}[theorem]{Example}
\newtheorem{assumptions}[theorem]{Assumptions}

\theoremstyle{remark}
\newtheorem{remark}[theorem]{Remark}

\newcommand{\E}[2][]{\mathbb{E}_{#1}\left[#2\right ]}

\renewcommand{\P}[2][]{\mathbb{P}_{#1}\left(#2\right )}
\newcommand{\R}{\mathbb{R}}

\newcommand{\N}{\mathbb{N}}
\newcommand{\ind}[1]{\mathds{1}_{#1}}

\renewcommand{\L}{\mathcal{L}}
\newcommand{\cadlag}{c\`adl\`ag~}

\newcommand{\Emeasure}{\xi}
\newcommand{\ExtendedE}{\bar{E}}
\newcommand{\law}{\operatorname{law}}
\newcommand{\disceval}[1]{\lfloor \frac{#1}{\Delta} \rfloor \Delta}
\newcommand{\discevaln}[1]{\lfloor \frac{#1}{\Delta_n} \rfloor \Delta_n}
\newcommand{\discevalh}[1]{\lfloor \frac{#1}{h} \rfloor h}
\newcommand{\indfunc}{\lambda}
\newcommand{\GammaDelta}{\Gamma_{\!\Delta}}
\newcommand{\GammaDeltan}{\Gamma_{\!\Delta_n}}

\newcommand{\Lmin}{\underline{\Lambda}}
\newcommand{\Xmin}{\underline{X}}
\newcommand{\taumin}{\underline{\tau}}
\renewcommand{\L}{\Lambda}
\newcommand{\GammaAl}{{\Gamma}^{\alpha}}
\newcommand{\GammaAlpert}{\tilde{\Gamma}^{\alpha}}

\newcommand{\Ldonskermin}{\underline{\L}}
\newcommand{\Ltime}{\underline{\L}}

\newcommand\mydots{\hbox to 1em{.\hss.\hss.}}

\usepackage{cancel}
\usepackage{soul}

\title{Implicit and fully discrete approximation of the supercooled Stefan problem in the presence of blow-ups}
\author{Christa Cuchiero\thanks{Vienna University, Department of Statistics and Operations Research, Data Science @ Uni Vienna, Kolingasse 14-16, A-1090 Wien, Austria, christa.cuchiero@univie.ac.at}
\and Christoph Reisinger \thanks{Mathematical Institute, University of Oxford, Andrew Wiles Building, Radcliffe Observatory Quarter, OX2 6GG, Oxford, U.K., christoph.reisinger@maths.ox.ac.uk}
\and Stefan Rigger \thanks{Vienna University, Department of Statistics and Operations Research, Kolingasse 14-16, A-1090 Wien, Austria, stefan.rigger@univie.ac.at.
\newline
The authors gratefully acknowledge financial support by the Vienna Science and Technology Fund (WWTF) under grant MA16-021 and  by the Austrian Science Fund (FWF) through grant Y 1235 of the START-program.
}}
\date{}

\begin{document}
\maketitle
\abstract{We consider two 
approximation schemes of the one-dimensional supercooled Stefan problem and prove their convergence, even in the presence of finite time blow-ups. 
All proofs are based on a probabilistic reformulation 
recently considered in the literature.
The first scheme is a version of the time-stepping scheme studied in 
\emph{V. Kaushansky, C. Reisinger, M. Shkolnikov, and Z. Q. Song, arXiv:2010.05281, 2020},
but here the flux over the free boundary and its velocity are coupled implicitly.
Moreover, we extend the analysis to more general driving processes than Brownian motion. The second scheme is a
Donsker-type approximation,  also interpretable as an implicit finite difference scheme, for which global convergence is shown under minor technical conditions. 
With stronger assumptions,
which apply in cases without blow-ups, we obtain additionally a convergence rate arbitrarily close to 1/2. Our numerical results suggest that this rate also holds for less regular solutions, in contrast to explicit schemes, and allow
a sharper resolution of the discontinuous free boundary in the blow-up regime. 
\section{Introduction}

The classical PDE formulation of the one-dimensional one-phase Stefan problem is
\begin{equation}
\label{u_PDE}
	\begin{split}
		& \partial_t u = {\frac{1}{2}} \partial_{xx}u,\;\;  x > \Lambda_t,\;\; t\ge 0, \\
		&u(t,\Lambda_t)=0,\;\; t\ge0,\\		
		& \dot{\Lambda}_t = \frac{\alpha}{2}\partial_x u(t,\Lambda_t),\;\;t\ge0,\\
		& u(0,x)=f(x),\;\; x\ge0 \quad\text{and} \quad \Lambda_0=0.
	\end{split}
\end{equation}
In physics terminology,
$-f$ is the initial temperature in a liquid, where we will consider $ f \geq 0$ corresponding to the supercooled regime; 
$-u(t,\cdot)$ is the temperature in the liquid phase at time $t$;
and  $\Lambda_t$ is the location of the liquid-solid boundary at time $t$.
The above relationship between the flux $\partial_x u$ at the free boundary $ \Lambda_t$
and the growth rate $\dot{\Lambda}_t$ of the solid phase is known as \emph{Stefan condition}. It
expresses the energy conservation at the interface,
in integrated form, 
\begin{equation*}
\Lambda_t =  \alpha\left(1- \int_{\Lambda_t}^{\infty} u(t,x) \, dx \right).
\end{equation*}

Variants of the Stefan problem, originally proposed in \cite{Stefan1}, have been studied in great detail in the applied mathematics literature. In particular, it has been established in a string of works starting in the 1970s (see, e.g., \cite{sherman1970general, herrero1996singularity, FP1,FP2,DF,HOL,LO,FPHO, Vis,DHOX,HX,Wei})
that the supercooled Stefan problem may exhibit finite time \emph{blow-ups}, whereby continuous solutions $t \mapsto \Lambda_t$ cease to exist. 

As closed-form solutions are rarely available for this type of free boundary problems, one typically has to resort to numerical methods.
We refer to \cite{caldwell2004numerical, mitchell2009finite} for a survey of numerical schemes. These works treat classical cases with different boundary conditions,
where  sign requirements (related to well-posedness without  blow-ups) on the initial or boundary data -- usually that the temperature of the material in the liquid phase is positive -- are imposed.
In the supercooled regime, where such conditions are not satisfied, regularization techniques
to prevent the formation of blow-ups in finite time  can be applied. For instance, in \cite{king2005regularization} the effect
of kinetic undercooling as a regularizing mechanism is analysed and it is shown how the (unregularized) supercooled Stefan problem 
can be recovered asymptotically. These asymptotics are also compared to numerical solutions of the unregularized problem.
The authors there follow (for both the regularized and the unregularized case) the so-called method of lines considered in \cite{fasano1986problem}.
More precisely, the continuous time PDE systems are discretized and solved at successive times as a sequence of free boundary
problems for  ordinary differential equations. They can in turn  be tackled via the so-called method of invariant embedding described in \cite{meyer1977one}, which gives rise to explicitely solvable Riccati equations.
Note that in the numerical studies of the unregularized case the blow-up
behaviour cannot be  fully reproduced due to truncation errors.
Further alternative (less physical) regularization approaches are discussed in \cite{fasano1990some}, albeit without a thorough numerical analysis.
In \cite{back2010numerical}, 
the ill-posed Stefan problem for melting a superheated solid, which is mathematically identical to the supercooled Stefan problem, is analysed. There, similarly as in \cite{king2005regularization}, the method of lines is applied and  existing numerical results from the literature in the blow-up regime are reproduced. 
In \cite{aiki1996blow}, blow-up points to one-phase Stefan problems, however with Dirichlet boundary conditions, are treated and studied numerically. 
Let us finally mention the recent paper  \cite{mccue2022traveling}, which analyses the Fisher--Stefan model, a generalization
of the well-known Fisher--KPP model, in the context of biological invasion, where
 the speed of
the moving boundary is related to the flux of population there. By rescaling this problem, it can be  compared to the supercooled Stefan problem. In this context, \cite{
mccue2022traveling} provides new links between these models and a numerical scheme for the Fisher--Stefan model.


Despite this plethora of articles on ill-posed Stefan problems, to our knowledge,
with the exception of the particle simulation scheme of \cite{kaushansky2020convergence}, no provably convergent algorithm is known for the
blow-up scenario. We contribute to this literature 
by proposing a class of numerical schemes for which we can prove global convergence, 
more specifically, that the discrete approximation of the free boundary converges to the true free boundary at all continuity points.

The following simple finite difference scheme is a canonical example of the schemes we will consider for \eqref{u_PDE}, where $u_i^k$ is an approximation to
$u(k h, i \sqrt{h})$ for $i,k>0$ and some fixed mesh widths $h>0$ in time and  $\sqrt{h}$ in space, respectively:
\begin{equation}
\label{impl_rec_intro}
\begin{split}
&\frac{u_i^k - u_i^{k-1}}{h}  =
\frac{1}{2} \frac{u_{i+1}^{k-1} - 2 u_i^{k-1} +  u_{i+1}^{k-1}}{(\sqrt{h})^2}, \;\; k>0,\, i > i_k, \\
& i_k = \bigg\lfloor
\frac{\alpha}{\Delta x} \bigg(1 - \sum_{i=i_k+1}^\infty u_i^k\bigg)
\bigg\rfloor, \;\; k >0, \\
& u_i^0 = \int_{i \sqrt{h}}^{(i+1) \sqrt{h}} f(x) \, dx, \;\; i \ge 0,
\quad\text{and}\quad
u_i^k = 0, \;\; k >0, i \le i_k.
\end{split}
\end{equation}
We will analyse and implement a scheme of type \eqref{impl_rec_intro} in Section \ref{sec:numerics},
with a slightly perturbed initial condition, as motivated later.
Note that \eqref{impl_rec_intro}
is reminiscent of the forward Euler approximation of the heat equation, but is nonetheless an implicit nonlinear equation
through the dependence of the discrete free boundary 
$i_k \sqrt{h}$ on the solution $(u_i^k)_i$ at time $k h$.

Our analysis is based on the following probabilistic reformulation of the problem:
\begin{equation}\label{eq:stefanproblem}
\left\{
\begin{aligned}
X_t &= X_{0-} + B_t - \L_t, \\ 
\tau &= \inf\{t \geq 0: X_t \leq 0 \}, \\
\Lambda_t &=\alpha\P{\tau \leq t},
\end{aligned}\right.
\end{equation}
where $X_{0-}$ is a real-valued non-negative random variable,
$\alpha > 0$, and $B$ is a standard Brownian motion independent of $X_{0-}$.
We will also study extensions where 
$B$ is replaced by a more general continuous-time process.

The link between \eqref{eq:stefanproblem} and \eqref{u_PDE} is found by noticing that over sufficiently small times the forward density 
$p(t,\cdot)$, $t\in[0,T]$ of the absorbed process $X_t\,\mathbf{1}_{\{\tau>t\}}$, $t\in[0,T]$ on $(0,\infty)$
 satisfies the initial-boundary value problem
\begin{equation}\label{p_PDE}
	\begin{split}
		& \partial_t p=\frac{1}{2}\partial_{xx}p+ \dot{\Lambda}_t\partial_x p,\;\;x\ge0,\;\;t\in[0,T], \\ 
	 	& p(0,x)=f(x),\;\; x\ge 0\quad\text{and}\quad p(t,0)=0,\;\;t\in[0,T], \\
		&\Lambda_t=\alpha\left( 1- \int_0^{\infty} p(t,x) dx\right), 
\;\;t\in[0,T],
	\end{split}
\end{equation}
as long as $X_{0-}$ has a sufficiently regular density $f$ supported on $[0,\infty)$.
Applying the  transformation $u(t,x):=  p(t,x-\Lambda_t)$, $x\ge \Lambda_t$ then leads to
\eqref{u_PDE} (see \cite{delarue2019global} for a rigorous proof).

The problem \eqref{eq:stefanproblem} and its variants have recently been used to model systemic risk of interconnected financial systems. There, $\Lambda_t$ models the proportion of banks defaulted at time $t$, and $\alpha$ the strength of interbank borrowing
(see e.g.~\cite{hambly2019mckean, nadtochiy2019particle, hambly2019spde, cuchiero2021optimal} for details and further references).
Equations with similar mean-field effects through thresholding derive from integrate-and-fire models of neurons as considered in \cite{delarue2015particle}.

Besides the applicability of \eqref{eq:stefanproblem} to systemic risk modeling and neuroscience, 
which is justified by certain propagation of chaos results (see \cite{delarue2015particle, cuchiero2020propagation}),
 the probabilistic formulation \eqref{eq:stefanproblem} has also several mathematical advantages. Most notably, it 
 allows for a rigorous definition of global solutions to \eqref{u_PDE}  even in the presence of blow-ups. For this it is however necessary to first establish appropriate solution concepts for the probabilistic version of \eqref{eq:stefanproblem}.
Indeed, due to the singular interaction, 
\eqref{eq:stefanproblem}
does not have unique solutions in general. Therefore, two physically and economically meaningful
notions of solutions have been developed: so-called \emph{physical solutions} and \emph{minimal solutions}.
A rigorous characterisation of unique physical solutions was given under certain technical assumptions on the initial conditions in \cite{delarue2019global}.
The  concept of  \emph{minimal solutions}, which are by definition unique, was in detail treated in \cite{cuchiero2020propagation}. In particular, it was shown there that the minimal solution is also a physical one. It is still an open
question if uniqueness of physical solutions holds in general, but 
it is clear that
the minimal solution coincides with the minimal physical one.

Among the numerical schemes introduced recently for the approximation of \eqref{eq:stefanproblem} are those using time-stepping (see \cite{kaushansky2020convergence}), 
iterations with heat kernels (see \cite{lipton2019semi}) and particle systems (see \cite{kaushansky2018simulation}), where in the latter the numerical schemes are based on the corresponding propagation of chaos result.

Motivated by promising numerical experiments based on implicit instead of explicit schemes, 
 we here first prove a global convergence result for an  \emph{implicit version of the Euler time-stepping scheme}, as considered in \cite{kaushansky2020convergence}.
Furthermore, we extend this result to
 more general drivers than Brownian motion and allow for processes that satisfy the so-called crossing property  (see \eqref{eq:crossingproperty}), which holds e.g.~for fractional Brownian motion with arbitrary Hurst parameter in $(0,1)$ and some non-degenerate continuous semi-martingales (see Example \ref{ex:cross}).

To further reduce the numerical complexity we then consider a scheme that -- in contrast to the time-stepping algorithm -- does not require a (time-consumimg) Monte Carlo simulation of the corresponding particle system. The main result in the second part of the paper is that convergence in the Brownian case still holds after a Donsker approximation of the driving Brownian motion, i.e.\ the increments over single time-steps are approximated by random variables matching the first two moments of the Brownian increments. Under more restrictive conditions, in particular ruling out blow-ups, a rate of 1/2 (modulo a log factor) is proven.

The computational significance of this result is that by use of suitable discrete random variables, the discrete measure is supported on a recombining binomial tree and can be computed deterministically by a recursion over time-steps without simulation.
The resulting discrete equations, of which \eqref{impl_rec_intro} is a special case, are then interpretable as approximation schemes (finite difference schemes) for \eqref{u_PDE}.
We refer to these as \emph{fully discrete} schemes, in contrast to the schemes based on timestepping approximations alone,
which require additional (particle) approximations.

We observe the following advantage in computational complexity.
If the step size in time is $h$, the spacing of the tree nodes is $h^{1/2}$. On a given time horizon and finite spatial domain
(by truncation of the real line), the total number of nodes, and hence computational complexity, is therefore $O(h^{-3/2})$.
Although we only prove an order of convergence in the case of no blow-ups, we find empirically that in all regimes studied (jumps, no jumps, different regularities)
the error is of order $h^{1/2}$; see Section \ref{sec:numerics} for details.
Hence, to achieve an error of $\epsilon$, the computational complexity is $O(\epsilon^{-3})$.
In contrast to that, for the Monte Carlo time-stepping scheme of \cite{kaushansky2018simulation, kaushansky2020convergence}, the simulation error using $n$ particles is 
$O(n^{-1/2})$, at cost $O(n)$, and the time stepping error in the most regular case is $h^{1/2}$. Setting $n$ proportional to $\lfloor h^{-1} \rfloor $, the computational complexity is then $O(\epsilon^{-4})$ to achieve an error of $\epsilon$.
Moreover, for low regularity, the convergence of the (explicit) time-stepping scheme is found to be arbitrarily slow in \cite{kaushansky2020convergence}, and hence the computational complexity for prescribed
accuracy arbitrarily high. For the  Donsker scheme, we do not observe such a phenomenon, but rather similar convergence rates in regular and irregular regimes. At the same time, the computational time is also considerably lower.

The remainder of the article is organized as follows. In Section \ref{sec::notation}, we introduce relevant notation applied throughout the paper. Section \ref{sec:time-stepping} is dedicated to the convergence result  of the implicit time-stepping scheme extended to more general driving processes than Brownian motion. In Section \ref{sec:donsker}, we then introduce the Donsker-type approximation scheme which has lower complexity and for which we can prove a rate of convergence in the regular case without blow-up. Section \ref{sec:numerics}
concludes with numerical tests showing a better resolution of jumps by implicit schemes compared to explicit schemes in the blow-up regime, particularly on coarse meshes, and a uniform convergence rate of 1/2 of the Donsker scheme across all scenarios, even where no such rate is proven.

\subsection{Notation}\label{sec::notation}
Throughout the paper, $D([-1,\infty))$ denotes the space of \cadlag functions on $[-1,\infty)$ endowed with the Skorokhod $M_1$-topology. For properties of the $M_1$-topology we refer to \cite{whitt2002stochastic} and \cite[Appendix A]{cuchiero2020propagation}. For $x \in D([-1,\infty))$ we denote by $\operatorname{Disc}(x)$ the set of discontinuity points of $x$. We also define the path functional $\indfunc_t(x) := \ind{(0,\infty)}(\inf_{0 \leq s \leq t} x_s)$, such that $\indfunc_t(x) = 0$ if $x_s > 0$ on $[0,t]$ and $\indfunc_t(x) = 1$ otherwise. The space of
continuous functions on $[0,\infty)$ is denoted by $C([0,\infty))$ and endowed with the topology of compact convergence, i.e., $w_n \to w$ in $C([0,\infty))$ if and only if $w_{n}\vert_{K} \to w\vert_{K}$ uniformly
for every compact $K \subseteq [0,\infty)$. We define the space of cumulative distribution functions on $[0,\infty]$ by
\begin{equation}\label{eq:DistFunctiondef}
M := \{ \ell \colon \overline{\R} \rightarrow [0,1]~|~  \ell \text{ \cadlag and increasing, }~\ell_{0-} = 0,~\ell_{\infty} = 1\},
\end{equation}
where $\overline{\R}$ is the two-point compactification of $\mathbb{R}$. We endow $M$ with the topology induced by the L\'evy metric, i.e., $\ell^n \to \ell$ in $M$ if and only if $\ell^n_t \to \ell_t$ for every $t \in [0,\infty]\setminus\operatorname{Disc}(\ell)$. This topology turns $M$ into a compact Polish space. Furthermore, we define $\ExtendedE:= C([0,\infty)) \times M$ and endow it with the product topology. For $\alpha >0$, we define $\iota_{\alpha} \colon \ExtendedE \rightarrow D([-1,\infty))$ for $w \in C([0,\infty))$ and $\ell \in M$ as 
\begin{align}\label{eq::iotadef}
\iota_{\alpha}(w,\ell)_t := \begin{cases} 
w_0 \quad & t \in [-1,0) \\
w_t - \alpha\ell_t \quad & t \in [0,\infty).
\end{cases}
\end{align}
Throughout the paper, we denote by $x$ a generic element of $D([-1,\infty))$, by $w$ a generic element of $C([0,\infty))$ and by $\ell$ a generic element of $M$.
If $S$ is a Polish space, we denote the space of probability measures on $S$ by $\mathcal{P}(S)$ and endow it with the topology of weak convergence,
i.e., we say that $\mu_n \to \mu$ in $\mathcal{P}(S)$ iff $\int_{S} F(x) \mathrm{d}\mu_n (x) \to \int_{S}F(x) \mathrm{d}\mu (x)$ for all $F \in C_b (S;\mathbb{R})$, where $C_b (S;\mathbb{R})$ denotes the space of continuous bounded functions from $S$ to $\mathbb{R}$.
If $\mu \in \mathcal{P}(S)$ and $F \colon S \rightarrow \mathbb{R}$, we denote the integral of $F$ with respect to $\mu$ also with brackets, i.e., we write 
$\int_{S}F(x)~\mathrm{d}\mu(x) = \langle \mu, F \rangle$. Furthermore, if $\nu$ is the pushforward of the measure $\mu$ with respect to the map $T$, we denote this by $T(\mu) = \nu$. In particular, with a slight abuse of notation, if $\xi$ denotes a measure on $\ExtendedE$, $\iota_{\alpha}(\xi)$ denotes its pushforward on $D([-1, \infty))$.

\section{Convergence of the time-stepping scheme}
\label{sec:time-stepping}
In this section we prove convergence of an implicit version of the time-stepping scheme considered in \cite{kaushansky2020convergence} for more general driving processes than Brownian motion.  
Specifically, consider the following McKean--Vlasov problem 
\begin{equation}\label{eq:mckeanproblem}
\left\{
\begin{aligned}
X_t &= X_{0-} + Z_t - \L_t, \\ 
\tau &= \inf\{t \geq 0: X_t \leq 0 \}, \\
\Lambda_t &=\alpha\P{\tau \leq t},
\end{aligned}\right.
\end{equation}
where $X_{0-} \in \R,~ \alpha > 0$, and $Z$ is a continuous 
stochastic process with $Z_0 = 0$ independent of $X_{0-}$ satisfying the following \emph{crossing property}, 
\begin{equation}\label{eq:crossingproperty}
\P{\tau < \infty, \inf_{0\leq s \leq h} (Z_{\tau+s}-Z_{\tau}) = 0} = 0, \quad h > 0,
\end{equation}
for every stopping time $\tau$ with respect to the natural filtration of $X_{0-}+Z$. Note that the crossing property
we use here is slightly more general than the one defined in \cite{delarue2015particle,cuchiero2020propagation,nadtochiy2020mean} since we require \eqref{eq:crossingproperty} to hold for any finite stopping time, not just for the hitting time of zero. Note also that \eqref{eq:crossingproperty} holds for every stopping time if and only if it holds for every bounded stopping time.

\begin{example}
\label{ex:cross}
\begin{enumerate}[a)]
\item Let $Z:= M + Y$, where $M$ is a continuous local martingale and $Y$ is 1/2-Hölder continuous on compacts. Suppose additionally that for every $K \in \mathbb{N}$ there is a strictly positive random variable $\epsilon_K$ such that $\langle M \rangle_t - \langle M \rangle_s \geq \epsilon_K (t-s)$ for $0 \leq s \leq t \leq K$, where $\langle M \rangle$ is the quadratic variation of $M$. Then, the proof of Theorem 3.5 in \cite{bender2012simple} shows that $Z$ satisfies the crossing property \eqref{eq:crossingproperty}. In particular, we may choose $Z$ to be Brownian motion.
\item The process $Z:=B^{H}$, where $B^H$ is fractional Brownian motion with Hurst parameter $H \in (0,1)$ satisfies the crossing property \eqref{eq:crossingproperty} by \cite[Theorem 1.1]{peyre2017fractional}.
\end{enumerate}
\end{example}

We start by establishing existence (and uniqueness) of minimal solutions which are defined as follows.
We call a solution $(\Xmin,\taumin,\Lmin)$ to the McKean--Vlasov problem \eqref{eq:mckeanproblem} \emph{minimal},
if for every solution $(X,\tau, \L)$ to \eqref{eq:mckeanproblem} we have
\begin{equation}\label{eq:minimaldef}
\Lmin_t \leq \L_t, \quad t \geq 0.
\end{equation}

We introduce the fixed-point operator $\Gamma$, defined for $\ell \in M$,
\begin{equation}\label{eqdef:Gamma}
\left\{
\begin{aligned}
X_t[\ell] &= X_{0-} + Z_t - \alpha \ell_t, \\ 
\tau[\ell] &= \inf\{t \geq 0: X_t[\ell] \leq 0 \}, \\
\Gamma[\ell]_t &= \P{\tau[\ell]\leq t}.
\end{aligned}\right.
\end{equation}
As the next proposition shows, $\Gamma$ is continuous on $M$.

\begin{proposition}\label{thm:gammacont}
The operator $\Gamma \colon M \rightarrow M$ is continuous.
\end{proposition}
\begin{proof}
Suppose that $\ell^n \to \ell$ in $M$. Define 
\begin{equation*}
\Emeasure^{n} := \law((X_{0-}+Z,\ell^n)), \quad
\Emeasure := \law((X_{0-}+Z,\ell)).
\end{equation*}
Then, $\Emeasure^{n} \to \Emeasure$ in $\mathcal{P}(\ExtendedE)$. Set $\mu_n := \iota_{\alpha} (\Emeasure^{n}),~ \mu := \iota_{\alpha}(\Emeasure)$ and note that $\langle \mu_n, \indfunc_t \rangle = \Gamma[\ell^n]_t,~ \langle \mu, \indfunc_t \rangle = \Gamma[\ell]_t,$ where $ \indfunc$ was defined in Section \ref{sec::notation}.

By Theorem 4.2 in \cite{cuchiero2020propagation} and the continuous mapping theorem, $\mu_n \to \mu$ in $\mathcal{P}(D([-1,\infty)))$. For $x \in D([-1,\infty))$ set $\tau_0(x) := \inf\{t\geq 0 \colon x_t \leq 0\}$ and define $\tau_{\ell}(w):= \inf\{t \geq 0 \colon w_t - \alpha \ell_t \leq 0\}$. Fixing $h>0$, we compute 
\begin{align}\label{eq::mucrossing}
& \mu\big(\big\{x \in D([-1,\infty)) \colon \tau_{0}(x) < \infty, \inf_{0\leq s\leq h} (x_{\tau_0+s} - x_{\tau_0}) = 0  \big\}\big) \\
&= \Emeasure\big(\big\{(w,\bar{\ell}) \in \ExtendedE\colon \tau_{0}(w-\alpha \bar{\ell}) < \infty, \inf_{0\leq s\leq h} [(w_{\tau_0+s} - w_{\tau_0}) -\alpha (\bar{\ell}_{\tau_0+s} - \bar{\ell}_{\tau_0})] = 0 \big\}\big) \notag \\
&\leq \Emeasure \big(\big\{(w,\bar{\ell}) \in \ExtendedE\colon \tau_{\ell} < \infty, \inf_{0\leq s\leq h} (w_{\tau_{\ell}+s} - w_{\tau_{\ell}}) = 0 \big\}\big) \notag \\
&= \P{\tau_{\ell} < \infty,\inf_{0 \leq s \leq h} (Z_{\tau_{\ell}+s} - Z_{\tau_{\ell}}) = 0} = 0 \notag,
\end{align}
where the inequality is due to the fact that $\bar{\ell} \in M$ is increasing.
Noting that the analogue of Lemma 5.5 in \cite{cuchiero2020propagation} holds with the assumption that \eqref{eq::mucrossing} vanishes for every $h>0$, it follows that
\begin{equation*}
\lim_{n\to\infty} \Gamma[\ell^n]= \lim_{n\to\infty} \langle \mu_n, \indfunc_{\cdot} \rangle = \langle \mu, \indfunc_{\cdot} \rangle = \Gamma[\ell].
\end{equation*}
\end{proof}

\begin{theorem}
There is a (unique) minimal solution to \eqref{eq:mckeanproblem}, and it is given by
\begin{equation}
\Lmin = \alpha\lim_{k\to\infty} \Gamma^{(k)}[0],
\end{equation}
where the convergence is understood as $\Gamma^{(k)}[0] \to \alpha^{-1}\Lmin$ in $M$.
\begin{proof}
This is analogous to the proof of Proposition 2.3 in \cite{cuchiero2020propagation}, making use of Proposition \ref{thm:gammacont} above.
\end{proof}
\end{theorem}





For $\Delta > 0$, define a time-discretized version $Z^{\Delta}$ by $Z^{\Delta}_t := Z_{\disceval{t}}$ for $t \geq 0$. In analogy to the situation in continuous time, we define a fixed-point operator and use it to construct the minimal solution in this time-discretized version of the McKean--Vlasov problem and prove continuity of the operator in a suitable sense.

\begin{lemma}\label{lem:gammadeltacont}
For $\ell \in M$, set
\begin{equation}\label{eqdef:GammaDelta}
\left\{
\begin{aligned}
X_t^{\Delta}[\ell] &= X_{0-} + Z_t^{\Delta} - \alpha \ell_{\disceval{t}},  \\ 
\tau^{\Delta}[\ell] &= \inf\{t \geq 0: X_t^{\Delta}[\ell] \leq 0 \}, \\
\GammaDelta[\ell]_t &= \P{\tau^{\Delta}[\ell]\leq t}.
\end{aligned}\right.
\end{equation}
Suppose that $\ell_{\disceval{\cdot}}^n \to \ell_{\disceval{\cdot}}$ in $M$ and that $\ell_0^n = \ell_0$. Assume in addition that either $\law(X_{0-})$ is atomless or that $\law(Z_t)$ is atomless for every $t > 0$. Then $\lim_{n \to \infty} \GammaDelta [\ell^n] = \GammaDelta[\ell]$ in $M$.
\end{lemma}
\begin{proof}
Note that the condition $\ell_{\disceval{\cdot}}^n \to \ell_{\disceval{\cdot}}$ in $M$ implies $\ell_{k\Delta}^{n} \to \ell_{k\Delta}$ for $k \in \N$. The assumption $\ell_0^n = \ell_0$ implies $\P{\tau^{\Delta}[\ell^n]= 0, \tau^{\Delta}[\ell]>0} = 0.$ For any $t \geq 0$,
\begin{align*}
\GammaDelta[\ell^n]_t - \GammaDelta[\ell]_t &\leq \P{\tau^{\Delta}[\ell^n] \leq t,~\tau^{\Delta}[\ell] >t} \\
& = \sum_{k=1}^{\lfloor\frac{t}{\Delta}\rfloor} \P{\tau^{\Delta}[\ell^n] = k\Delta, \tau^{\Delta}[\ell]>t} \\
&\leq \sum_{k=1}^{\lfloor\frac{t}{\Delta}\rfloor} \P{ X_{k\Delta}^{\Delta}[\ell^n]\leq 0,~ X_{k\Delta}^{\Delta}[\ell] >0  } \\ 
&= \sum_{k=1}^{\lfloor\frac{t}{\Delta}\rfloor} \P{X_{0-}+Z_{k\Delta} - \alpha \ell_{k\Delta}^n \leq 0,~ X_{0-} + Z_{k\Delta} - \alpha \ell_{k\Delta} > 0} \\
&=\sum_{k=1}^{\lfloor\frac{t}{\Delta}\rfloor} \P{X_{0-} + Z_{k\Delta} \in (\alpha\ell_{k\Delta},\alpha\ell_{k\Delta}^n]}.
\end{align*}
Now if $\law(Z_t)$ is atomless for every $t>0$, using the independence of $Z$ and $X_{0-}$ we may rewrite this as
\begin{align*}
\sum_{k=1}^{\lfloor\frac{t-\Delta}{\Delta}\rfloor} \mathbb{E} \left[\P{x + Z_{k\Delta} \in (\alpha\ell_{k\Delta},\alpha\ell_{k\Delta}^n]})\vert_{x = X_{0-}} \right],
\end{align*}
and the dominated convergence theorem yields that $\limsup_{n\to\infty} \GammaDelta[\ell^n]_t \leq \GammaDelta[\ell]_t.$ The same argument works if we assume that $\law(X_{0-})$ is atomless, as we see by noticing that $\P{X_{0-} \!\! + \! Z_{k\Delta} \in (\alpha\ell_{k\Delta},\alpha\ell_{k\Delta}^n]} = \E{\P{X_{0-} \!\! + \! z \in (\alpha\ell_{k\Delta},\alpha\ell_{k\Delta}^n]} \! \vert_{z=Z_{k\Delta}}}$. Interchanging the roles of $\ell^n$ and $\ell$ in the estimates then yields the claim.
\end{proof}

\begin{definition}\label{def:discretized}
We say that $\Lambda^{\Delta}$ solves the \emph{discretized McKean--Vlasov problem} if $\alpha\Gamma_{\Delta}[\alpha^{-1}\Lambda^{\Delta}] = \Lambda^{\Delta}$. If $\underline{\Lambda}^{\Delta}$ is such that $\underline{\Lambda}^{\Delta} \leq \Lambda^{\Delta}$ for every $\Lambda^{\Delta}$ that solves the discretized McKean--Vlasov problem, we call $\underline{\Lambda}^{\Delta}$ \emph{minimal}. We shall refer to the scheme as \emph{implicit time-stepping scheme}.
\end{definition}

\begin{remark}\label{rem:implicit_explicit}
With Definition \ref{def:discretized}, $\Lambda^{\Delta}$ is a solution if and only if
\begin{align} \label{eq::lambdadiscretized}
\Lambda_{k\Delta}^{\Delta} = \alpha \mathbb{P}\left(\min_{0 \leq i \leq k} \{X_{0-} + Z_{i\Delta} - \Lambda_{i\Delta}^{\Delta}\} \leq 0\right), 
{}\end{align}
for every $k \in \mathbb{N}$. Since $\Lambda_{k\Delta}^{\Delta}$ appears 
on both sides of \eqref{eq::lambdadiscretized}, this is an implicit 
equation for $\Lambda^{\Delta}$ (the solution of which is not unique in general),
therefore we call our time-stepping scheme implicit. If we define an alternative notion of solution through \eqref{eq::lambdadiscretized} by taking the minimum on the right-hand side over $\{0,\dots,k-1\}$ instead of $\{0,\dots,k\}$, this gives a recursion for $\Lambda^{\Delta}$, and we obtain the time-stepping scheme of \cite{kaushansky2020convergence},
which we refer to as explicit in the following.
\end{remark}

\begin{proposition}\label{thm::DeltaMinimal}
There is a (unique) minimal solution of the discretized McKean--Vlasov problem and it is given by
\begin{align}\label{eq:DeltaIteration}
\underline{\Lambda}^{\Delta} = \alpha\lim_{k\to\infty} \Gamma_{\Delta}^{(k)}[0].
\end{align}
\end{proposition}
\begin{proof}
$\Gamma_{\Delta}$ is monotone in the sense that if $\ell^1 \leq \ell^2$ then $\Gamma_{\Delta}[\ell^1] \leq \Gamma_{\Delta}[\ell^2]$. Therefore,
\begin{equation*}
0 \leq \GammaDelta[0] \leq \GammaDelta^{(2)}[0] \leq \dots
\end{equation*}
This allows us to define $\hat \ell_t^{\Delta} :=  \lim_{k\to\infty} \GammaDelta^{(k)}[0]_t$. Then $\tilde{\ell}_t^{\Delta}:= \hat{\ell}_{t+}^{\Delta}$ lies in $M$ and by construction $\lim_{k\to\infty}\GammaDelta^{(k)}[0] = \tilde{\ell}^{\Delta}$ in $M$. Note that $\GammaDelta^{(k)}[0]$ is a step function with jumps at times $\Delta \N$, and so the same holds for $\tilde{\ell}^{\Delta}.$ This implies that $\GammaDelta^{(k)}[0]_{\disceval{t}} = \GammaDelta^{(k)}[0]_t$ and $\tilde{\ell}_{\disceval{t}}^{\Delta} = \tilde{\ell}_t^{\Delta}$ for $t \geq 0$. As $\GammaDelta^{(k)}[0]_s = \GammaDelta[0]_s$ for $0 \leq s < \Delta$, it follows that $\tilde{\ell}^{\Delta}_0 = \GammaDelta^{(k)}[0]_0$ for every $k \in \N$, and therefore we may apply Lemma \ref{lem:gammadeltacont} to find
\begin{equation*}
\GammaDelta[\tilde{\ell}^{\Delta}] = \GammaDelta[\lim_{k\to\infty}\GammaDelta^{(k)}[0]] = \lim_{k\to\infty}\GammaDelta[\GammaDelta^{(k)}[0]] = \tilde{\ell}^{\Delta},
\end{equation*}
so $\tilde{\Lambda} := \alpha\tilde{\ell}^{\Delta}$ is a solution of the 
discretized McKean--Vlasov problem. If $\Lambda^{\Delta}$ is another solution, then since $0 \leq \alpha^{-1}\Lambda^{\Delta}$, and $\Gamma^{\Delta}$ is monotone we have $\alpha\Gamma[0] \leq \alpha\Gamma[\alpha^{-1} \Lambda^{\Delta}] = \Lambda^{\Delta}$. Proceeding inductively we obtain $\alpha\Gamma^{(k)}[0] \leq \Lambda^{\Delta}$ and therefore $\tilde{\L}^{\Delta}_t \leq \Lambda^{\Delta}_t$ for every continuity point of $\tilde{\L}^{\Delta}$. By right-continuity, it follows that $\tilde{\L}^{\Delta}$ is minimal.  \end{proof}

\begin{remark}
In contrast to the explicit scheme in \cite{kaushansky2020convergence}, the time-stepping scheme we propose to solve here requires us to solve the 
implicit condition that $\underline{\Lambda}^{\Delta}$ is the minimal solution of \eqref{eq::lambdadiscretized} through the iteration \eqref{eq:DeltaIteration}. The analogous results we prove here hold as well for the explicit version of the scheme, but we expect that the explicit version smoothes out the jumps while the implicit scheme should not, 
and therefore the implicit scheme should lead to
a more accurate approximation of jump sizes and jump times. We observe this empirically in 
Section 4.
\end{remark}

We are now ready to state the main result of this section.

\begin{theorem}\label{thm:nestedconvergence}
Choose a sequence $\Delta_n > 0$ such that the resulting discretizations are nested, i.e., $\Delta_n \mathbb{N} \subseteq \Delta_{n+1}\mathbb{N}$ for all $n \in \mathbb{N}$, and such that $\lim_{n\to\infty} \Delta_n$ = 0. Assume in addition that either $\law(X_{0-})$ is atomless or that $\law(Z_t)$ is atomless for every $t>0$. Then $\lim_{n\to\infty} \Ltime^{\Delta_n} = \Lmin$ in $M$.
\end{theorem}

We collect some properties of the space $M$ in the next lemma.

\begin{lemma}\label{lem:Mconvergence}
Let $\ell^n, \ell \in M$. The following statements are equivalent.
\begin{enumerate}[(a)]
\item $\ell^n \to \ell$ in $M$.
\item Let $t$ be a continuity point of $\ell$. Then, for every $\epsilon > 0$, there is a $\delta > 0$ with
\begin{equation}
\limsup_{n\to\infty} \sup_{s \in [t-\delta,t+\delta]} |\ell_s^n - \ell_s| \leq \epsilon.
\end{equation}
\item There is a co-countable set $D \subseteq [0,\infty)$ such that $\ell_t^n \to \ell_t$ for $t \in D$.
\end{enumerate}
\end{lemma}

\begin{proof}
$\left[(a) \implies (b)\right]$: Fix $\epsilon > 0$. As $\ell$ is continuous at $t$ and $\ell$ has 
at most countably many discontinuity points, there is $\delta > 0$ such that both $t+\delta$
and $t-\delta$ are continuity points and $|\ell_s - \ell_r| < \frac{\epsilon}{2}$ for $s,r \in [t-\delta, t+\delta]$.
Choose $n_0 \in \N$ large enough such that $|\ell_{t+\delta}^{n} - \ell_{t+\delta}| < \frac{\epsilon}{2}$ and $|\ell_{t-\delta}^{n} - \ell_{t-\delta}| < \frac{\epsilon}{2}$ for all $n \geq n_0$. Then, for $s \in [t-\delta, t+\delta]$ we have
\begin{align*}
\ell_s^n - \ell_s &\leq \ell_{t+\delta}^{n} - \ell_s < \ell_{t+\delta} - \ell_s + \frac{\epsilon}{2} < \epsilon, \\
\ell_s^n - \ell_s &\geq \ell_{t-\delta}^{n} - \ell_s > \ell_{t-\delta} - \ell_s - \frac{\epsilon}{2} > -\epsilon,
\end{align*}
proving that $\sup_{s \in [t-\delta,t+\delta]} |\ell_s^n - \ell_s| \leq \epsilon$ for $n \geq n_0$. The reverse implication is clear.\\

$\left[(c) \implies (a)\right]$: By compactness of $M$, after passing to a subsequence if necessary, we may assume that $\ell^n \to \bar{\ell}$ for some $\bar{\ell} \in M$. Since $D$ is co-countable, the set $D\setminus \operatorname{Disc}(\bar{\ell})$ is co-countable as well, and therefore dense in $[0,\infty)$. For $t \in D\setminus \operatorname{Disc}(\bar{\ell})$, we obtain $\ell_t = \lim_{n\to\infty} \ell_t^n = \bar{\ell}_t$, and by right-continuity it follows that $\bar{\ell} = \ell$. Hence $\ell^n \to \ell$ in $M$. The reverse implication is clear.
\end{proof}

\begin{proposition}\label{prop:GammaDeltauniformconv}
Let $\Delta_n \to 0$ and suppose that $\ell^n \to \ell$ in $M$. Then,
\begin{equation}
\lim_{n\to\infty} \GammaDeltan[\ell^n] = \Gamma[\ell].
\end{equation}
\end{proposition}

\begin{proof}
Step 1: We first show that $(X_{0-}+Z^{\Delta_n},\ell_{\discevaln{\cdot}}^{n}) \to (X_{0-}+Z,\ell)$ in $\ExtendedE$. Fix $T>0$. For $\omega \in \Omega$ and $\epsilon > 0$, by uniform continuity there almost surely is $\delta > 0$ such that $|Z_{t}(\omega) - Z_{s}(\omega)| < \epsilon$ whenever $|s-t|<\delta$ and $s,t \in [0,T]$. For $\Delta_n < \delta$, we then have 
\begin{equation*}
|Z_{t}(\omega) - Z_{t}^{\Delta_n}(\omega)| < \epsilon, \quad t \in [0,T],
\end{equation*}
and hence $\lim_{n\to\infty} \|Z - Z^{\Delta_n}\|_{C([0,T])} = 0$ almost surely. 
Let $t$ be a continuity point of $\ell$. Write $g^n := \ell_{\discevaln{\cdot}}^{n}$. Fix $\epsilon > 0$, and choose $\delta > 0$ as in Lemma \ref{lem:Mconvergence}$.(b)$. We then have
\begin{equation*}
\limsup_{n\to\infty}|g_t^n - \ell_t| \leq \limsup_{n\to\infty} \sup_{s\in [t-\delta,t+\delta]}|\ell_s^n - \ell_s| + \limsup_{n\to\infty} |\ell_{\discevaln{t}} - \ell_t| \leq \epsilon. 
\end{equation*}
As $\epsilon > 0$ was arbitrary, we see that $g^n \to \ell$ in $M$. \\

Step 2: Setting $\Emeasure_n := \law((X_{0-}+Z^{\Delta_n},g^n))$ and $\Emeasure := \law((X_{0-}+Z,\ell))$, by Step 1 we have $\Emeasure_n \to \Emeasure$ in $\mathcal{P}(\ExtendedE)$. By the same reasoning as in the proof of Proposition~\ref{thm:gammacont} we obtain that $\langle \iota_\alpha (\Emeasure^n),\indfunc_{\cdot}\rangle \to \langle \iota_\alpha (\Emeasure), \indfunc_{\cdot}\rangle$ in $M$ with $\indfunc$ being defined in Section \ref{sec::notation}.
 Since $\GammaDeltan[\ell^n]_t = \langle \iota_\alpha (\Emeasure^n),\indfunc_{t}\rangle$ and $\Gamma[\ell]_t = \langle \iota_\alpha (\Emeasure), \indfunc_{t}\rangle$, this yields the claim.
\end{proof}

We are now prepared to prove  Theorem \ref{thm:nestedconvergence}.

\begin{proof}[Proof of Theorem \ref{thm:nestedconvergence}]

A straightforward induction using Proposition \ref{prop:GammaDeltauniformconv} shows that $\lim_{n\to\infty} \GammaDeltan^{(k)}[0] = \Gamma^{(k)}[0]$. Since $\Delta_n \mathbb{N} \subseteq \Delta_{n+1} \mathbb{N}$, it holds that $\Gamma_{\Delta_n}[\ell] \leq \Gamma_{\Delta_{n+1}}[\ell]$ for any $\ell \in M$. Set $J:= \operatorname{Disc}(\Lmin) \cup \{\operatorname{Disc}(\Gamma^{(k)}[0])~|~ k \in \N\} \cup \bigcup_{n\geq 0} \Delta_n \N$. Then $J$ is countable, and for $t \notin J$ we have by Proposition \ref{thm::DeltaMinimal}
\begin{align*}
\Lmin_t = \alpha\lim_{k\to\infty} \Gamma^{(k)}[0]_t = \alpha\lim_{k\to\infty} \lim_{n\to\infty} \GammaDeltan^{(k)}[0]_t = \alpha\lim_{n\to\infty} \lim_{k\to\infty} \GammaDeltan^{(k)}[0]_t = \lim_{n\to\infty}  \Ltime^{\Delta_n},
\end{align*}
where we may exchange the order of the limits because $\GammaDeltan^{(k)}[0]_t$ is increasing both in $n$ and in $k$. Lemma \ref{lem:Mconvergence}$.(c)$ yields the claim.
\end{proof}

\section{A Donsker-type approximation}
\label{sec:donsker}

In this section, we consider the numerical approximation of a special case of the McKean--Vlasov problem \eqref{eq:mckeanproblem} with  $Z = B$, with $B$ a Brownian motion independent of $X_{0-}$. The idea is to replace the Brownian motion by its Donsker approximation, i.e., for $h>0$ \footnote{To distinguish the two schemes we use here $h$ rather than $\Delta$ for the time step.} we consider

\begin{equation}\label{eq:mckeandonskerproblem}
\left\{
\begin{aligned}
X_t^{h} &= X_{0-}^{h} + B_t^{h} - \L_t^{h}, \\ 
\tau^{h} &= \inf\{t \geq 0: X_t^{h} \leq 0 \}, \\
\Lambda_t^{h} &=\alpha\P{\tau^{h} \leq t},
\end{aligned}\right.
\end{equation}
where
\begin{equation}
B_{t}^h = \sqrt{h}\sum_{k=0}^{\lfloor t/h \rfloor} Y_{k},
\end{equation}
 and $(Y_i)_{i\in\N}$ is an i.i.d.\ sequence of random variables satisfying $\mathbb{E}[Y_1] = 0$ and $\mathbb{E}[Y_1^2]=1$.
We repeat the same reasoning as previously, obtaining the minimal solution of \eqref{eq:mckeandonskerproblem} through a fixed-point iteration. We call a solution $\underline{\L}^h$ of \eqref{eq:mckeandonskerproblem} minimal, if $\underline{\L}^h \leq \L^h$ for every $\L^h$ that solves \eqref{eq:mckeanproblem}. The minimal solution will then give rise to the implicit Donsker approximation scheme presented in Section \ref{sec:numerics}. 

\begin{lemma}\label{lem:Gammahfixedpoint}
Assume that $X_{0-}^h$ is atomless. Define the operator $\Gamma_h$ via
\begin{align*}
X_t^{h}[\ell] &= X_{0-}^h + B_t^{h} - \alpha \ell_{\discevalh{t}}, \\ 
\tau^{h}[\ell] &= \inf\{t \geq 0: X_t^{h}[\ell] \leq 0 \}, \\
\Gamma_h[\ell]_t &:=\P{\tau^{h}[\ell] \leq t}.
\end{align*}
Then it holds that
\begin{align}\label{eq::fixedpointdonsker}
\Ldonskermin^h =  \alpha \lim_{k\to\infty} \Gamma_h^{(k)}[0].
\end{align}
\end{lemma}
\begin{proof}
By the same arguments given in Lemma \ref{lem:gammadeltacont}, if $\ell_{\discevalh{\cdot}}^n \to \ell_{\discevalh{\cdot}}$ in $M$,
then $\lim_{n\to\infty} \Gamma_h [\ell^n] = \lim_{n\to\infty} \Gamma_h [\ell]$ in $M$. Reasoning as in the proof of Theorem \ref{thm:nestedconvergence}, using Step 1 of Proposition \ref{prop:GammaDeltauniformconv}, we obtain \eqref{eq::fixedpointdonsker}.
\end{proof}

\begin{remark}\label{rem:finitetimeminimal}
If we were to pose problem \eqref{eq:mckeandonskerproblem} on a finite time horizon $[0,T]$, the analogous
result to Lemma \eqref{lem:Gammahfixedpoint} would hold by the same reasoning. In particular, since the solution operator
for the finite time horizon is the restriction to $[0,T]$ of $\Gamma_h$, the representation \eqref{eq::fixedpointdonsker} shows
that the minimal solution for the problem on $[0,T]$ is simply the restriction of the global minimal solution. 
\end{remark}

From here on, we will denote by $\Ldonskermin^{h}$ the minimal solution in the Donsker approximation with initial condition 
\begin{equation} \label{eq::init} X_{0-}^{h} = 
\left(\left\lfloor \frac{X_{0-}}{\sqrt{h}} \right\rfloor + \left\lfloor\log(h)^2 \right\rfloor\right)\sqrt{h}.
\end{equation}
The reasons for this choice are twofold: first, to make the theoretical results applicable to the numerical scheme of Section \ref{sec:numerics}, we replace $X_{0-}$ by a space-discretized version of itself, given by $\left\lfloor \frac{X_{0-}}{\sqrt{h}} \right\rfloor \sqrt{h}$; second,
for technical reasons, we require a perturbation term of order $\mathcal{O}(\sqrt{h}\log(h)^2)$ in the initial condition.\\

 The goal of this section is to show that, with the above choice of initial condition \eqref{eq::init}, 
  $\Ldonskermin^{h} \to \Lmin$ in $M$ as $h \to 0$. We require some additional assumptions. 

\begin{assumptions} \label{as:Donsker} We assume that
\begin{enumerate}[(i)]	
\item $X_{0-}$ has a bounded density, which we denote by $V_{0-}$. \label{as:bounded_density}
\item The moment generating function of $Y_1$ exists in a neighbourhood of $0$, i.e., $\mathbb{E}[\exp(uY_1)] < \infty$ for some $\delta > 0$ and $|u| < \delta$.
\label{as:mgf}
\end{enumerate}
\end{assumptions}

\subsection{Convergence results}
On an abstract level, the proof of convergence has two main parts: (i) proving that limit points of solutions to the Donsker problem are solutions to the McKean--Vlasov problem; (ii) identifying limit points of minimal solutions of the Donsker problem as the minimal solution of the McKean--Vlasov problem. 
Typically, part (ii) is the more challenging one. Fortunately, a similar approach as in \cite{cuchiero2020propagation}
works.

The central idea of this approach is to find a sequence of optimization problems that allows us to relate the minimal solution of the (perturbed) Donsker problem to the minimal solution of the McKean--Vlasov equation, showing that the latter dominates the former asymptotically as $h \to 0$.

The following lemma constitutes the main technical ingredient of this proof. As a rough intuition, setting $\ell = \Lmin$, the lemma allows us to quantify how far away the minimal solution of the McKean--Vlasov problem is from being a solution to the Donsker problem when $h$ is small.

\begin{lemma}\label{lem:Gammahconvergencerate}
Suppose that Assumptions \ref{as:Donsker} are satisfied. Then, there is a constant $C > 0$ depending on $T$ and $\|V_{0-}\|_{\infty}$ such that for all sufficiently small $h > 0$ it holds 
\begin{equation}\label{eq:Gammahconvergencerate}
\sup_{\ell \in M}\sup_{t \in [0,T]} |\Gamma_h [\ell]_t - \Gamma[\ell]_t | \leq C \sqrt{h}\log(T/h),
\end{equation}
where $\Gamma_h$ is defined as in Lemma \ref{lem:Gammahfixedpoint} with $X_{0-}^h = \left\lfloor\frac{X_{0-}}{\sqrt{h}} \right \rfloor \sqrt{h}$.
\end{lemma}

The proof of the above lemma requires some preparations.
The convergence rate in \eqref{eq:Gammahconvergencerate}
is essentially the convergence rate of the Donsker approximation with respect to the Prokhorov metric, which we transfer to the estimate above through a Lipschitz mapping theorem.

\begin{definition}\label{def:prokhorov}
Let $(S,d)$ be a Polish space. For $\mu, \nu \in \mathcal{P}(S)$ define the Prokhorov distance between $\mu$ and $\nu$ as
\begin{equation}\label{eq:prokhorovmetric}\rho_{P}^{S}(\mu,\nu) = \inf\{\epsilon > 0 : \mu(B) \leq \nu(B^\epsilon) + \epsilon ~\text{for all Borel sets } B\},\end{equation}
where $B^\epsilon = \{x \in S \colon d_S(x,B) < \epsilon\}$. If $X,Y$ are random variables on $S$ we will write $\rho_P^{S}(X,Y)$ in place of $\rho_P(\law(X),\law(Y))$. 
\end{definition}
It is well-known that the Prokhorov metric metrizes weak convergence (in the probabilistic sense) and $(\mathcal{P}(S),\rho_P^{S})$ is a Polish space. By \cite[Theorem 3.4.2]{whitt2002stochastic}, Lipschitz maps preserve the Prokhorov metric in the sense that $\rho_P^{S}(g(X),g(Y)) \leq (1\vee K) \rho_P^{S} (X,Y)$ for any $K$-Lipschitz map $g$. We will need a slight generalization of this result, which allows $g$ to depend on an additional random variable that is independent of the Lipschitz component. 

\begin{theorem}\label{thm:lipschitzmapping}
Let $(S,d),(S',d'),(R,d_R)$ be Polish spaces and suppose that $g \colon S' \times S \rightarrow R$ is $K$-Lipschitz in the second variable, i.e., 
\begin{equation}
d_R(g(x,y),g(x,z)) \leq K d_S(y,z), \quad x \in S',~ y,z \in S.
\end{equation}
Suppose furthermore that the random variables $X,Y,Z$ satisfy $X \perp Y$, $X\perp Z$. Then,
\begin{equation}
\rho_P^R (g(X,Y),g(X,Z)) \leq (1 \vee K) \rho_P^{S}(Y,Z).
\end{equation}
\end{theorem}

\begin{proof}
We may assume wlog that $K \geq 1$. By the Strassen-Dudley representation theorem \cite[Corollary 11.6.4]{dudley2018real}, there is a probability space $(\tilde{\Omega}, \tilde{\mathbb{P}})$ with random variables $\tilde Y, \tilde Z$ satisfying $\law(\tilde Y) = \law(Y), \law(\tilde Z) = \law(Z)$ such that $\rho_P^{S}(Y,Z) = \rho_K(\tilde Y, \tilde Z)$, where $\rho_{K}$ is the Ky-Fan metric, i.e.,
\begin{equation*}
\rho_K(\tilde Y, \tilde Z) = \inf\{\epsilon > 0 \colon \tilde{\mathbb{P}}(d(\tilde Y,\tilde Z)> \epsilon) \leq \epsilon\}.
\end{equation*} 
Now set $\overline{\Omega} = \Omega \times \tilde{\Omega}$ and $\overline{\mathbb{P}} = \mathbb{P} \otimes \,\tilde{\mathbb{P}}$, then
$\law((X,\tilde Y)) = \law((X,Y))$ and $\law((X,\tilde Z)) = \law((X,Z))$ by construction of $(\overline{\Omega},\overline{\mathbb{P}})$ and the independence assumption. In particular, we have $\rho_P^{R}(g(X,Y),g(X,Z)) = \rho_P^{R}(g(X,\tilde Y), g(X,\tilde Z))$. Noting the general fact that $\rho_P \leq \rho_K$, we consider $\epsilon > 0$ such that $\tilde{\mathbb{P}}(d(\tilde Y,\tilde Z) > \epsilon/K) \leq \epsilon / K$. Then,
\begin{equation*}
\overline{\mathbb{P}}(d_R(g(X,\tilde Y),g(X,\tilde Z)) > \epsilon) \leq \overline{\mathbb{P}}(d(\tilde Y, \tilde Z) > \epsilon/K) \leq \epsilon/K \leq \epsilon.
\end{equation*}
It follows that
\begin{equation*} \rho_K (g(X,\tilde Y),g(X,\tilde Z)) \leq \inf\{\epsilon > 0 \colon \tilde{\mathbb{P}}(d(\tilde Y, \tilde Z)>\epsilon /K) \leq \epsilon /K\} = K \rho_K(\tilde Y, \tilde Z).
\end{equation*}
We have obtained that
\begin{equation*}
\rho_P^{R}(g(X,Y), g(X,Z)) \leq \rho_K(g(X,\tilde Y), g(X,\tilde Z)) \leq K \rho_K(\tilde Y, \tilde Z) = K \rho_P^{S}(Y,Z).
\end{equation*}
\end{proof}

\begin{example}
We illustrate with an example that Theorem \ref{thm:lipschitzmapping} does not hold in general without the independence assumption. We consider $S = S' = R = \mathbb{R}$, we set $\law((X,Y)) = \frac{1}{2} \delta_{(0,0)} + \frac{1}{2} \delta_{(1,1)},$ and $\law((X,Z)) = \frac{1}{2} \delta_{(0,1)} + \frac{1}{2} \delta_{(1,0)}$. Then, $\law(X) = \law(Y) = \law(Z) = \frac{1}{2} \delta_0 + \frac{1}{2} \delta_1$, so in particular $\rho^S_P (Y,Z) = 0$. However, taking $g(x,y) := |x-y|$, we see that $\law(g(X,Y)) = \delta_0$ and $\law(g(X,Z)) = \delta_1$, so $\rho^R_P(g(X,Y),g(X,Z)) = 1.$ 
\end{example}

\begin{corollary}\label{cor:kolmogorovlipschitzmapping}
Suppose $g:S' \times S \rightarrow \mathbb{R}$ is $K$-Lipschitz in the second variable. Suppose that $X \perp Y, X \perp Z$ and that the cdf of
$g(X,Y)$ is Lipschitz with constant $L$. Then, for all $x \in \mathbb{R}$ we have
\begin{equation}|\mathbb{P}(g(X,Y) \leq x) - \mathbb{P}(g(X,Z)\leq x)| \leq (1+L)(1 \vee K)\rho_P^{R} (Y,Z).
\end{equation}
\end{corollary}

\begin{proof}
Denote the cdf of $g(X,Y)$ by $F$ and that of $g(X,Z)$ by $G$. It holds that
\begin{align*}
\sup_{x \in \mathbb{R}} |G(x) - F(x)| &\leq (1+L) \inf\{\epsilon > 0 \colon F(x-\epsilon) - \epsilon \leq G(x) \leq F(x+\epsilon) + \epsilon\, ~\forall x \in \mathbb{R}\} \\ &\leq (1+L) \rho_P^{R} (g(X,Y),g(X,Z)).
\end{align*}
The claim now follows from Theorem \ref{thm:lipschitzmapping}.
\end{proof}

We state two technical lemmata for future reference.

\begin{lemma}\label{lem:infestimate}
Let $x^1,x^2$ be \cadlag paths. Then,
\begin{align*}
|\inf_{0\leq s \leq t} x_s^1 - \inf_{0\leq s \leq t} x_s^2| \leq \sup_{0 \leq s \leq t} |x_s^1 - x_s^2|
\end{align*}
\end{lemma}
\begin{proof}
Let $\epsilon > 0$ and choose $r \in [0,t]$ such that $ x_{r}^2 \leq \inf_{0\leq s\leq t} x_s^2 + \epsilon$. Then,
\begin{align*}
\inf_{0\leq s \leq t} x_s^1 - \inf_{0\leq s \leq t} x_s^2 \leq x_{r}^1 - x_{r}^{2} + \epsilon \leq \sup_{0\leq s \leq t}|x_{s}^1 -x_{s}^2| + \epsilon.
\end{align*}
Reversing the roles and noting that $\epsilon > 0$ was arbitrary proves the claim.
\end{proof}

\begin{lemma}\label{lem::X0bounds}
Let $X_{0-}$ admit a bounded density, denoted by $V_{0-}$ and let $Y^1,Y^2$ be \cadlag and such that $X_{0-}$ is independent of $(Y^1,Y^2)$.
Then, it holds that
\begin{align*}
\left|\mathbb{P}\left(\inf_{0\leq s\leq t} X_{0-}+Y_{s}^1\leq 0\right) - \mathbb{P}\left(\inf_{0\leq s\leq t}X_{0-}+Y_s^2 \leq 0\right)\right|
 \leq \|V_{0-}\|_{\infty} \E{\sup_{0\leq s \leq t} |Y_s^1 - Y_s^2|}.
\end{align*} 
\end{lemma}
\begin{proof}
Set $I := [-\inf_{s \leq t} Y_s^1, -\inf_{s \leq t}, Y_s^2] \cup [-\inf_{s \leq t} Y_s^1, -\inf_{s \leq t}, Y_s^2]$ and note that
by Lemma \ref{lem:infestimate}, the length of $I$ is bounded by $\sup_{0\leq s \leq t} |Y_s^1 - Y_s^2|$. We argue as in the proof of Lemma 2.1 in \cite{ledger2020uniqueness}, using the assumption on $X_{0-}$ and the general inequality $\mathbb{P}(A) - \mathbb{P}(B) \leq \mathbb{P}(A\setminus B)$,
\begin{align*}
&\left|\mathbb{P}\left(\inf_{0\leq s\leq t} X_{0-}+Y_{s}^1\leq 0\right) - \mathbb{P}\left(\inf_{0\leq s\leq t}X_{0-}+Y_s^2 \leq 0\right) \right|\\
&\leq \mathbb{P}\left(X_{0-} \in I\right) \\
&= \E{\int_{0}^{\infty} \ind{I}(x) V_{0-}(x)~dx} \\
&\leq \|V_{0-}\|_{\infty} \E{\sup_{0\leq s \leq t} |Y_s^1 - Y_s^2|}. 
\end{align*}
\end{proof}

\begin{proof}[Proof of Lemma \ref{lem:Gammahconvergencerate}]
Step 1: We transfer the convergence rate on $[0,1]$ to $[0,T]$ for $T>0$. Set 
\begin{equation}
S := \left\{x \in D[0,1] \colon \operatorname{Disc}(x) \subseteq h \mathbb{Q} \right\}, 
\quad R := \left\{x \in D[0,T] \colon \operatorname{Disc}(x) \subseteq \frac{h}{T}\mathbb{Q}\right\}.
\end{equation}
In words, $S$ is the set of \cadlag paths on $[0,1]$ that can only jump at times that are rational multiples of $h$.
Note that if we endow the sets $S$ and $R$ with the uniform metric, they are separable and hence Polish spaces.  
By Theorem 1.16 in \cite{csorgo1993weighted}\footnote{Note that in \cite{csorgo1993weighted}, the result is stated to hold for all $h>0$, but the proof reveals that it only holds for sufficiently small $h>0$}, we have $\rho_P^S (B^{h'}, W) \leq C \sqrt{h'}\log(1/h')$ for sufficiently small $h' >0$. Define $f: S \rightarrow R$ by $f(x) = \sqrt{T} x_{\cdot/T}$, then $f$ is $\sqrt{T}$-Lipschitz. Theorem $\ref{thm:lipschitzmapping}$ and the self-similarity of the Donsker approximation and Brownian motion give
\begin{align*}
\rho_P^R (B^h,W) &= \rho_P^R (\sqrt{T} B_{\cdot/T}^{h/T},\sqrt{T} W_{\cdot/T}) = \rho_P^R(f(B^{h/T}),f(W)) \\
&\leq C\frac{(1 \vee \sqrt{T})}{\sqrt{T}}\sqrt{h}\log(T/h),
\end{align*}
for $h>0$ sufficiently small, which concludes Step 1. \\

Step 2: Fix $\ell \in M$ and define the map $g_t : \mathbb{R} \times R \rightarrow \mathbb{R}$ by $g_t(x_0,x) := \inf_{0 \leq s \leq t} x_0 + x_s - \alpha \ell_s$. 
Note that $g(x_0,\cdot)$ is 1-Lipschitz for every $x_0 \in \mathbb{R}$ by Lemma \ref{lem:infestimate}. Denote the cdf of $g(X_{0-},B)$ by $F$. We check that $F$ is Lipschitz continuous: for $x,y \in \mathbb{R}$, setting $Y^1:= B - \alpha \ell - x$ and $Y^2:= B - \alpha \ell - y$ in Lemma \ref{lem::X0bounds}, we obtain

\begin{equation}
\label{eq:Lipschitz}
|F(x) - F(y)| \leq \|V_{0-}\|_{\infty} |x-y|.
\end{equation}
By Corollary \ref{cor:kolmogorovlipschitzmapping}, we therefore obtain for $t \in [0,T]$ and $h >0$ small enough
\begin{align*}
 |\mathbb{P}(g_t (X_{0-},B^h) \leq 0) - \mathbb{P}(g_t (X_{0-},B) \leq 0)| 
\leq C (1+\|V_{0-}\|_{\infty})\frac{(1 \vee \sqrt{T})}{\sqrt{T}}\sqrt{h}\log(T/h).
\end{align*}
Since with $X_{0-}^{h} := \left\lfloor\frac{X_{0-}}{\sqrt{h}} \right \rfloor \sqrt{h}$
it holds that  ${|X_{0-} - X_{0-}^h| \leq \sqrt{h}}$ almost surely, setting $Y^1 := -\sqrt{h} + B^h - \alpha \ell$ and $Y^2 := B^h - \alpha \ell$ in Lemma \ref{lem::X0bounds}, we find
\begin{align*}
 &\mathbb{P}(g_t (X_{0-}^{h},B^h) \leq 0) - \mathbb{P}(g_t (X_{0-},B^h) \leq 0) \\
&\leq \mathbb{P}(g_t (X_{0-}-\sqrt{h},B^h) \leq 0) - \mathbb{P}(g_t (X_{0-},B^h) \leq 0) \\
&\leq \|V_{0-}\|_{\infty} \sqrt{h},
\end{align*}
and arguing analogously for the downwards estimate we find
\begin{align}\label{eq:gthestimate}
 |\mathbb{P}(g_t (X_{0-}^{h},B^h) \leq 0) - \mathbb{P}(g_t (X_{0-},B^h) \leq 0)| \leq \|V_{0-}\|_{\infty} \sqrt{h}.
\end{align}
Noting that $\Gamma_h[\ell]_t = \mathbb{P}(g_t (X_{0-}^{h},B^h) \leq 0)$ and $\Gamma[\ell]_t = \mathbb{P}(g_t (X_{0-},B) \leq 0)$, taking the supremum over $t \in [0,T]$ and absorbing some terms into the constant $C$, the claim follows.
\end{proof}

\begin{lemma}\label{lem:asymptoticallydominated}
Let $\Ldonskermin^h$ be the minimal solution to \eqref{eq:mckeandonskerproblem} with initial condition given by $\eqref{eq::init}$ and suppose that Assumptions \ref{as:Donsker} are satisfied. Then, for any $t > 0$, 
\begin{equation}
\limsup_{h \to 0} \Ldonskermin_t^h \leq \Lmin_t.
\end{equation}
\end{lemma}
\begin{proof}
Define the sequence of optimal values
\begin{align}
V_t^h = \inf_{\ell \in M} c_h(\ell), \quad  c_h(\ell) = \alpha \ell_t + \frac{\alpha}{ \sqrt{h}\lfloor\log(h)^2\rfloor} \sup_{s \in [0,t]} |\Gamma_h[\ell]_s - \ell_s|,
\end{align}
where $\Gamma_h$ is defined as in Lemma \ref{lem:Gammahfixedpoint}.  \\

Step 1: We show that $V_t^h$ is asymptotically dominated by $\Lmin$ as $h \to 0$. By Lemma \ref{lem:Gammahconvergencerate}, for small enough $h>0$ we have
\begin{align}
V_t^h \leq c_h(\alpha^{-1}\Lmin) = \Lmin_t + \frac{\alpha}{\sqrt{h}\lfloor\log(h)^2\rfloor}  \sup_{s \in [0,t]} | \Gamma_h[\Lmin]_s - \Lmin_s| \leq \Lmin_t + \alpha C \frac{\log(t/h)}{\lfloor\log(h)^2\rfloor}
\end{align}
and therefore $\limsup_{h\to 0} V_t^h \leq \Lmin_t$. \\

Step 2: We show that $V_t^h$ asymptotically dominates $\Ldonskermin_t^h$. Choose a sequence $(L^h)_{h \geq 0}$ with $L^h \in M$ such that $c_h(L^h) - V_t^h \to 0$ as $h \to 0$. Suppose that $\sup_{s \in [0,t]} |\Gamma_h[L^h]_s - L_s^h| > \alpha^{-1} h^{1/2} \lfloor \log(h)^2 \rfloor$ for some $h$, then we have $c_h(L^h) > c_h(\alpha^{-1}\tilde \L^h)$, where $\tilde \L^h$ is a solution to \eqref{eq:mckeandonskerproblem} with $X_{0-}^{h} := \left\lfloor\frac{X_{0-}}{\sqrt{h}} \right \rfloor \sqrt{h}$. We may therefore assume wlog that $\sup_{s \in [0,t]} |\Gamma_h[L^h]_s - L_s^h| \leq \alpha^{-1} h^{1/2} \lfloor \log(h)^2 \rfloor$ for all $h > 0$. For the sake of brevity, we write $d_h : = \alpha^{-1} h^{1/2} \lfloor \log(h)^2 \rfloor$ in the following. We then have
\begin{equation}\label{eq:lowerboundinequality}
L_s^h \geq \Gamma_h[L^h]_s - d_h \quad s \in [0,t].
\end{equation}
For $\ell \in M$ define the operator 
\begin{align}\label{eq:Gammatilde}
\tilde{\Gamma}_h[\ell]_t := \mathbb{P}\left(\inf_{0\leq s \leq t} \left[ \left(\left\lfloor \frac{X_{0-}}{\sqrt{h}} \right\rfloor + \left\lfloor\log(h)^2 \right\rfloor\right)\sqrt{h} +
  B_s^h - \alpha \ell_s \right] \leq 0\right).
\end{align}
 Inequality \eqref{eq:lowerboundinequality} and the monotonicity of $\Gamma_h$ imply
\begin{equation}
L_s^h \geq \Gamma_h [-d_h]_s -d_h = \tilde \Gamma_h[0]_s - d_h, \quad s \in [0,t].
\end{equation}
Substituting this again into inequality \eqref{eq:lowerboundinequality} yields
\begin{align*}
L_s^h \geq \Gamma_h[\tilde \Gamma_h [0] - d_h]_s - d_h = \tilde\Gamma_h^{(2)}[0]_s - d_h, \quad s \in [0,t].
\end{align*}
Proceeding inductively, we obtain $L_s^h \geq \tilde\Gamma_h^{(k)}[0]_s - d_h$ for every $s \in [0,t]$ and $k \in \N$, and taking the limit as $k \to \infty$, by Lemma \ref{lem:Gammahfixedpoint} we obtain $\alpha L_s^h \geq \Ldonskermin_s^h - \alpha d_h$ for $s \in [0,t].$ Since $d_h \to 0$ as $h \to 0$, combined with the definition of $c_h$ and Step 1, this yields
\begin{equation*}
\limsup_{h\to 0} \Ldonskermin_t^h \leq \limsup_{h\to 0} \alpha L_t^h \leq \limsup_{h\to 0} c_h(L^h) = \limsup_{h\to 0} V_t^h \leq \Lmin_t.
\end{equation*}
\end{proof}

\begin{theorem}\label{th:maindonsker}
Suppose that Assumptions \ref{as:Donsker} are satisfied. Then, the sequence of (perturbed) minimal solutions to the Donsker problem converges to the minimal solution to the McKean--Vlasov problem, i.e., it holds that
\begin{equation}
\lim_{h\to 0} \underline{\L}_t^h = \Lmin_t,
\end{equation} 
for every $t \geq 0$ that is a continuity point of $\underline{\Lambda}$.
\end{theorem}

\begin{proof} Choose a sequence $(h_n)_{n\geq 1}$ such that $h_n > 0$ and $\lim_{n\to\infty} h_n = 0$. By compactness of $M$, after passing to 
a subsequence if necessary, we may assume that $\Ldonskermin^{h_n} \to \Lambda$ for some $\Lambda$ with $\alpha^{-1}\Lambda \in M$. If $x$ is a c\`adl\`ag 
path, we introduce the notation 
$\GammaAl[x] := \alpha\Gamma[\alpha^{-1}x]$, such that $\L$ solves \eqref{eq:mckeanproblem} iff $\GammaAl[\L]=\L$. We use the notation $\GammaAl_h$ in the analogous way. We show that $\Lambda$ solves the McKean-Vlasov problem.
By Theorem \ref{thm:gammacont}, $\Gamma$ is continuous, so for $t \notin \operatorname{Disc}(\L) \cup \operatorname{Disc}(\Gamma[\L])$ we have
\begin{align*}
|\GammaAl[\L]_t - \L_t| &= \lim_{n\to\infty} |\GammaAl[\Ldonskermin^{h_n}]_t - \Ldonskermin_t^{h_n}| \\
 &\leq \limsup_{n\to\infty}\left[ |\GammaAl[\Ldonskermin^{h_n}]_t - \GammaAl_{h_n}[\Ldonskermin^{h_n}]_t| + |\GammaAl_{h_n}[\Ldonskermin^{h_n}]_t - \GammaAlpert_{h_n}[\Ldonskermin^{h_n}]_t|  \right].
\end{align*}
where $\GammaAlpert[x]$ is defined as $\alpha\tilde{\Gamma}_{h_n}[\alpha^{-1}x]$ with $\tilde{\Gamma}_{h}$ defined in \eqref{eq:Gammatilde} and we used that $\Ldonskermin^{h_n} = \GammaAlpert[\Ldonskermin^{h_n}]$  since $\Ldonskermin^{h_n}$ is a solution to the perturbed Donsker problem. Due to Lemma \ref{lem:Gammahconvergencerate}, the first term above vanishes as $n \to \infty$. For the second term, setting $Y^1:= \sqrt{h_n}(1+\lfloor \log(h_n)^2 \rfloor) + B^h - \Ldonskermin^{h_n}$ and $Y^2:= -\sqrt{h_n}+ B^h - \Ldonskermin^{h_n}$, we have
\begin{align*}
\alpha\mathbb{P}\left(X_{0-} + \inf_{s \leq t} Y_s^1 \leq 0\right) \leq \GammaAlpert_{h_n}[\Ldonskermin^{h_n}]_t \leq \GammaAl_{h_n}[\Ldonskermin^{h_n}]_t \leq 
\alpha\mathbb{P}\left(X_{0-} + \inf_{s \leq t} Y_s^2 \leq 0\right),
\end{align*}
so by Lemma \ref{lem::X0bounds} it follows that
\begin{align*}
 |\GammaAl_{h_n}[\Ldonskermin^{h_n}]_t - \GammaAlpert_{h_n}[\Ldonskermin^{h_n}]_t| \leq \alpha\|V_{0-}\|_{\infty} \sqrt{h_n}(2+\lfloor \log(h_n)^2 \rfloor). 
 \end{align*} Therefore, $\GammaAl[\L]_t = \L_t$. By right-continuity of $\L$ and $\GammaAl[\L]$, we find $\L = \GammaAl[\L]$. By definition, this implies
$\L \geq \Lmin$. On the other hand, for $t \notin \operatorname{Disc}(\L)$, Lemma \ref{lem:asymptoticallydominated} yields
\begin{align*}
\L_t = \limsup_{n\to\infty} \underline{\L}_t^{h_n} \leq \Lmin_t.
\end{align*}
By right-continuity, we find $\L \leq \Lmin$ and hence $\L = \Lmin$.  \end{proof}

In the case $\alpha \|V_{0-}\|_{\infty} <1$, we also get a rate of convergence as shown in the next theorem. By \cite[Theorem 2.2 and the comment below]{ledger2020uniqueness}, this corresponds to the weak feedback regime, where uniqueness and in particular continuity of solutions holds true.

\begin{theorem}
Suppose that Assumptions \ref{as:Donsker} are satisfied and that additionally $\alpha \|V_{0-}\|_{\infty} <1$.
Then, there is a constant $C$ depending on $T, \alpha$ and $\|V_{0-}\|_{\infty}$ such that
\[
 \sup_{t \in [0,T]}|\Lmin_t -\underline{\L}^h_t| \leq C \log(h)^2\sqrt{h},
\]
for $h>0$ sufficiently small.
\end{theorem}

\begin{proof}
Applying Lemma \ref{lem::X0bounds} with $Y^1 = B - \alpha \ell^1$ and $Y^2 = B - \alpha \ell^2$ shows
\begin{align}\label{eq:GammaLipschitz}
\sup_{t \in [0,T]} | \Gamma[\ell^1]_t - \Gamma[\ell^2]_t| \leq \alpha \|V_{0-}\|_{\infty} \sup_{t \in [0,T]} |\ell_t^1 - \ell_t^2|.
\end{align} 
Recalling the notation of the proof of Theorem \ref{th:maindonsker}, we have $\Ldonskermin^{h_n} = \GammaAlpert_{h_n}[\Ldonskermin^{h_n}]$ , since $\Ldonskermin^{h_n}$ is a solution to the perturbed Donsker problem and $\GammaAl [\Lmin]= \Lmin$. Moreover, 
\begin{align*}
 \sup_{t \in [0,T]}|\Lmin_t -\underline{\L}^h_t|&= \sup_{t \in [0,T]}|\GammaAl[\Lmin]_t -\GammaAlpert_{h_n}[\underline{\L}^h]_t|\\
 & \leq \sup_{t \in [0,T]}|\GammaAl[\Lmin]_t -\GammaAl[\underline{\L}^h]_t|+ \sup_{t \in [0,T]}|\GammaAl[\Lmin^h]_t -\GammaAl_{h}[\underline{\L}^h]_t| \\
 &\quad + \sup_{t \in [0,T]}|\GammaAl_{h}[\Lmin^h]_t -\GammaAlpert_{h}[\underline{\L}^h]_t|\\
 &\leq \alpha\|V_{0-}\| \sup_{t \in [0,T]}|\Lmin_t -\underline{\L}^h_t| + \alpha C \sqrt{h} \log(T/h) \\
 &\quad+ \alpha\| V_{0-}\|_{\infty} \sqrt{h} (2 + \lfloor \log(h)^2 \rfloor),
\end{align*}
where the last estimate follows from \eqref{eq:GammaLipschitz}, Lemma \ref{lem:Gammahconvergencerate} and the third term is estimated as in the proof of Theorem \ref{th:maindonsker}. Rearranging terms and absorbing some into the constant $C$, the claim follows.
\end{proof}

\begin{remark}
Note that the choice $\sqrt{h}\lfloor \log(h)^2 \rfloor$ for the perturbation of the initial condition was somewhat arbitrary, since the proofs of the previous theorems work for any function $\delta(h)$ that converges to zero slowly enough such that 
\begin{align*}\lim_{h \to 0} \frac{\sqrt{h}\log(T/h)}{\delta(h)} = 0
\end{align*} for every $T > 0$. In particular, we could consider $\sqrt{h} \lfloor \epsilon\log(h)^2 \rfloor$ for any $\epsilon > 0$, which suggests that the perturbation can be ignored in practice, and we will indeed ignore it in the numerical tests in the next section. It seems natural to conjecture that the result should also hold with $\epsilon = 0$, and this would indeed follow from Conjecture 6.10 in \cite{cuchiero2020propagation}, however we do not have a proof at this point.
\end{remark}

\section{Numerical tests}\label{sec:numerics}

In this section, we first describe our implementation of the Donsker scheme from Section \ref{sec:donsker} and then present results
from numerical tests. We also compare them with the implicit and explicit time-stepping scheme.

\subsection{A tree-type scheme for the Donsker approximation}

We consider
\eqref{eq:mckeandonskerproblem}
and specify  $Y_k$  to be Rademacher random variables, i.e.\ $\mathbb{P}(Y_k=-1) = \mathbb{P}(Y_k=1) = 1/2$.
We introduce a time mesh $t_k = k h$, $k\ge 0$ an integer, and
\begin{eqnarray*}
u_i^k &=& \mathbb{P}(X^h_{t_k} = i \sqrt{h} - \Ldonskermin^h_{t_k}, \tau^h>t_k),
\end{eqnarray*}
for $i \in \mathbb{Z}$, using the notation defined in Section \ref{sec:donsker}, from which it follows
\begin{eqnarray*}
\Ldonskermin^h_{t_k} = \alpha\left(1 - \sum_{i=i_k+1}^\infty u_i^k\right), \qquad i_k = \bigg\lfloor  \frac{ \Ldonskermin^h_{t_k}}{\sqrt{h}} \bigg\rfloor.
\end{eqnarray*}

We have that
\begin{eqnarray}
\label{impl_rec}
u_i^k &=& \left\{
\begin{array}{rl}
\frac{1}{2} u_{i-1}^{k-1} + \frac{1}{2} u_{i+1}^{k-1}, & \qquad i > i_k + 1, \\
\frac{1}{2} u_{i+1}^{k-1}, & \qquad i = i_k + 1, \\
0, & \qquad i < i_k + 1,
\end{array}
\right.
\end{eqnarray}
for $k>0$. This is an implicit scheme as $u^k$ and $\Ldonskermin^h$ are implicitly coupled. The recurrence relation \eqref{impl_rec} has the resemblance of a binomial tree, shifted by the interaction term, and
can be rearranged into the finite difference scheme \eqref{impl_rec_intro} for \eqref{u_PDE}. It
can also be interpreted as a special type of  semi-Lagrangian scheme (see \cite{camilli1995approximation, falcone2013semi}) for the forward equation \eqref{p_PDE}.

\subsection{Iterative solution}
\label{sec:iterative}

Note that \eqref{impl_rec} is a non-linear equation through the dependence of $i_k$ on $u^k$ via $\Ldonskermin^h_{t_k}$.
Lemma \ref{lem:Gammahfixedpoint} suggests a fixed-point iteration to solve simultaneously for $\Ldonskermin^h_{t_k}$ and the vector $u^k$
for each $k$. We assume that we know the cdf of $X_{0-}$ exactly, and therefore we can calculate, 
 $$\mathbb{P}(X^h_{0-} = i\sqrt{h}) = \mathbb{P}(X_{0-} \in [i\sqrt{h},(i+1)\sqrt{h})),$$ 
for all $i \in \mathbb{Z}$. To calculate $u^{0}_i$ for $i \in \mathbb{Z}$, we need to determine $\Ldonskermin_{0}^{h}$ first, which we obtain as in Lemma \ref{lem:Gammahfixedpoint}
through the iteration initialized with $\lambda^{0} = 0$ and
\begin{align*} \lambda^{n+1} = \alpha\sum_{j=0}^{\iota^{n}} \mathbb{P}(X_{0-}^{h} = j\sqrt{h}), 
\quad \iota^{n} =  \bigg \lfloor \frac{\lambda^{n}}{\sqrt{h}} \bigg \rfloor.
\end{align*}
The iteration terminates when $\iota^{n+1} = \iota^{n}$, which happens after at most $\lfloor \alpha/\sqrt{h}\rfloor$ iterations, since $\lambda^{n} \leq \lambda^{n+1}$ and hence $\iota^{n} \leq \iota^{n+1}$ as well. This yields $\Ldonskermin_0^{h}$. We then calculate $u_i^0$ via $$u_i^{0} = 0 \text{ for } i \leq  \bigg \lfloor \frac{\Ldonskermin_0^{h}}{\sqrt{h}} \bigg \rfloor \quad  \text{and} \quad  u_i^{0} = \mathbb{P}(X_{0-}^{h} = i\sqrt{h})  \text{ for } i >  \bigg \lfloor \frac{\Ldonskermin_0^{h}}{\sqrt{h}} \bigg \rfloor. $$ The vectors $u^k = (u_i^k)_i$ are then calculated recursively through \eqref{impl_rec},
using a local in time version of the iteration in Lemma \ref{lem:Gammahfixedpoint} that iterates only over the scalar loss at each time point.
Set $\lambda^0 = \Ldonskermin^h_{t_{k-1}}$, $\iota^0 = \lfloor \lambda^0 / \sqrt{h} \rfloor$,
and for $n\ge 0$,
\begin{eqnarray}
\label{it_lam}
u_i^{k,n+1} &=& \left\{
\begin{array}{rl}
\frac{1}{2} u_{i-1}^{k-1} + \frac{1}{2} u_{i+1}^{k-1}, & \qquad i > \iota^{n} + 1, \\
\frac{1}{2} u_{i+1}^{k-1}, & \qquad i = \iota^n + 1, \\
0, & \qquad i < \iota^n + 1,
\end{array}
\right. \\
\nonumber
\lambda^{n+1} &=& \alpha\left(1 - \sum_{i=\iota^n+1}^\infty 
u_i^{k,n+1}\right), \qquad 
\iota^{n+1} =  \bigg \lfloor \frac{ \lambda^{n+1}}{\sqrt{h}} \bigg \rfloor.
\end{eqnarray}
Expressing this in terms of $\Gamma_h$, if we define
\begin{align}
\hat{\Lambda}^{(0)} := \begin{cases}
\Ldonskermin_t^{h}, \quad &t \in [0,t_{k-1}], \\
\Ldonskermin_{t_{k-1}}^h, \quad &t \in [t_{k-1},t_{k}] 
\end{cases}
\end{align}
and set $\hat{\Lambda}^{(n+1)} = \Gamma_h[\hat{\Lambda}^{(n)}]$, then $\lambda^{n+1}$ 
above is equal to $\Gamma_h[\hat{\Lambda}^{(n)}]_{t_k}$. We convince ourselves that this
computes the minimal solution: As in the case of the initial condition, the iteration is increasing in $n$ for the losses, $\lambda^{n+1} \ge \lambda^{n}$ and $\iota^{n+1} \ge \iota^{n}$, and the iteration terminates in at most $\lfloor \alpha / \sqrt{h} \rfloor$ iterations
(because $\iota^n$ can only take values in $\{i_{k-1},\ldots, \lfloor \alpha/ \sqrt{h} \rfloor\}$).
If $n_0$ is the smallest $n$ such that $\lambda^{n+1} = \lambda^{n}$, then $\hat{\Lambda}^{(n_0)}$
solves \eqref{eq:mckeandonskerproblem} on $[0,t_k]$, and therefore $\Lambda^{(n_0)}_t \geq \Ldonskermin_t^h$
for $t \in [0,t_k]$ by Remark \ref{rem:finitetimeminimal}. On the other hand, since $\Ldonskermin^h$ is increasing, it holds
that $\hat{\Lambda}^{(n_0)} \leq \Ldonskermin^{h}$, and by the monotonicity of $\Gamma_h$ and
a straightforward induction it follows that $\hat{\Lambda}^{(n)} \leq \Ldonskermin^{h}$
for $n \in \mathbb{N}$ and hence $\hat{\Lambda}_t^{(n_0)} = \Ldonskermin_t^{h}$ for $t \in [0,t_k]$.

\subsection{Explicit and implicit particle scheme}
\label{sec:particle}

Here, we briefly discuss the implementation of the time-stepping scheme specified by \eqref{eqdef:GammaDelta} and 
Definition \ref{def:discretized}.

The particle method in \cite{kaushansky2020convergence}, which we refer to here as explicit particle scheme,
is given  by, for a fixed number $n\ge 1$ of particles,
\begin{eqnarray}
\label{x_part}
X^{\Delta, n, (m)}_{k \Delta} &=& X_{0-} + Z_{k \Delta}^{n, \Delta, (m)} - \Lambda_{k \Delta}^{\Delta, n}, 
\qquad \quad k \ge 0, \; 1\le m \le n,
\\
\Lambda_{k\Delta}^{\Delta, n} &=& \frac{\alpha}{n} \sum_{m=1}^n
\mathbbm{1}_{\left\{\min_{0\le i<k} X^{\Delta, n, (m)}_{i \Delta} \le 0 \right\} }, \quad k > 0, \;\;\; \Lambda_{0}^{\Delta, n} = 0.
\label{l_part_exp}
\end{eqnarray}
We define a corresponding implicit particle approximation by \eqref{x_part} but with
$\Lambda_{k\Delta}^{\Delta, n} $ the smallest solution to
\begin{eqnarray}
\Lambda_{k\Delta}^{\Delta, n} &=& \frac{\alpha}{n} \sum_{m=1}^n
\mathbbm{1}_{\left\{
\left(\min_{0\le i < k}  X^{\Delta, n, (m)}_{i \Delta} \right)
{ \text{\normalsize $\wedge$}}
\left( X_{0-} + Z_{k \Delta}^{n, \Delta, (m)} - \Lambda_{k \Delta}^{\Delta, n}\right)
 \le 0 \right\} }, \quad k\ge 0,
 \label{l_part_imp}
\end{eqnarray}
where we interpret the minimum over an empty set as $\infty$.
Similar to Section \ref{sec:iterative} for
the fully discrete scheme, $\Lambda^\Delta$ is implicitly defined and can be found iteratively:
\begin{eqnarray*}
\Lambda_{k\Delta}^{\Delta, n, (0)} \!\!&\!=&\! \frac{\alpha}{n} \sum_{m=1}^n
\mathbbm{1}_{\left\{\min_{0\le i<k} X^{\Delta, n, (m)}_{i \Delta} \le 0 \right\} }, \quad k \ge 0, \\
\Lambda_{k\Delta}^{\Delta, n, (j)} \!\!&\!=&\!
\frac{\alpha}{n} \sum_{m=1}^n
\mathbbm{1}_{\left\{
\left(\min_{0\le i < k}  X^{\Delta, n, (m)}_{i \Delta} \right)
{ \text{\normalsize $\wedge$}}
\left( X_{0-} + Z_{k \Delta}^{n, \Delta, (m)} - \Lambda_{k \Delta}^{\Delta, n, (j-1)}\right)
 \le 0 \right\} }, \quad k \ge 0, \; j > 0.
\end{eqnarray*}
It is clear that $\Lambda_{k\Delta}^{\Delta, n, (j)}$ is increasing in $j$, and that it terminates in finitely many iterations in a fixed-point, which has to be the minimal solution.

%
%

\subsection{Numerical results}

In this section, we analyse computational  aspects and, especially, the numerical accuracy of the scheme.

As a first example, we consider a $\Gamma(k,\theta)$ distribution for
 $X_{0-}$ with $k=2$ and $\theta=1/3$. Note that the initial density $V_{0-}$ is globally Lipschitz in this case.
 For simplicity, we here set $u_i^0 = V_{0-}(i \sqrt{h})$.
 Furthermore, we fix the time interval $[0,T]=[0,0.02]$. 

We first examine iteration \eqref{it_lam} for the Donsker scheme. For two different values of $\alpha$ and an increasing number of time points $N$ such that $h=1/N$, Table \ref{table:iterations} gives the number of iterations before termination, first averaged over all time steps, and then the maximum number of iterations for any time step.

\begin{table}[ht!]
\centering
\begin{tabular}{|l|l|c|c|c|c|c|c|}
\hline
&time points $N$ & 100 & 200 & 400 & 800 & 1600 & 3200 \\ \hline \hline
$\alpha=0.5,$ &av.\ iter. & 1.0200  &  1.0150  &  1.0125  &  1.0088  &  1.0063  &  1.0047 \\
 & max.\ iter. &  2  &   2   &  2  &   2   &  2  &   2   \\  \hline
$\alpha=1.5,$ &av.\ iter.  &  1.2800  &  1.1800   & 1.1175   & 1.0775   & 1.0513 &   1.0341 \\
&max.\ iter.  & 16  &  18  &  20  &  20 &   22   & 23 \\  
\hline
\end{tabular}
\caption{Average and maximum number of iterations over all $N$ time points
for $\alpha=0.5$ (no jump) and $\alpha=1.5$ (jump).
}
\label{table:iterations}
\end{table}

In the regular case without a jump ($\alpha=0.5$), there are never more than 2 iterations needed, while the average number is close to 1. This is explained by 
the fact that
$i_k \neq i_{k-1}$ only if $\Ldonskermin^h_{t_{k}} - \Ldonskermin^h_{t_{k-1}} > \alpha \sqrt{h}$,
so in the regular regime, where $\Ldonskermin_t$ is differentiable, a change of $i_k$ will only happen every $O(1/\sqrt{h})$ time steps,
and usually by only 1. 
Put differently, because of the monotonicity of $\Ldonskermin^h_{t_{k}}$ in $k$ and that of the iteration, the total number of iterations summed up over all time steps
is bounded by $\alpha/ \sqrt{h}$.
So the average number of iterations 
 is 1 + $O(\sqrt{h})$.

In the presence of a jump ($\alpha=1.5$), a larger number of iterations is needed at the time of the jump, but this number only grows mildly under mesh refinement, and
on average the number of iterations is still close to 1.

In either case, therefore, the computational cost of the iteration amounts to less than 10\% of the overall cost for reasonably fine time meshes.

Figure \ref{fig:loss} shows the numerical free boundary $(t_k,\Ldonskermin^h_{t_k}/\alpha)$, labeled `implicit',
and scaled by $1/\alpha$ to show how much mass was absorbed at the boundary. 
The solution is compared to an approximation computed with a simplified scheme where the iteration is stopped after the first iteration.
The latter can be considered an explicit treatment where the interaction term computed at each time point is used to compute
the density at the following point, and is hence labeled `explicit', in contrast to scheme \eqref{impl_rec} where $p^k$ and $\Ldonskermin^h_{t_k}$ are implicitly coupled.
This reveals that the explicit scheme smoothes out the jump and takes a significant time for the losses to `catch up'.
In contrast, the implicit scheme \eqref{impl_rec} reproduces a sharp jump, i.e., for sufficiently small $h$ the increment $\Ldonskermin^h_{t_{k+1}}-\Ldonskermin^h_{t_k}$ is small for all but one $k$, while this largest increase does not go to 0 as $h$ diminishes, but converges to the true jump size.\footnote{We thank Andreas S\o{}jmark for a discussion on the implicit treatment of jumps.}

\begin{figure}[t!]
    \centering
    \begin{subfigure}[t]{0.49\textwidth}
            \hspace{-0.2cm}
        \includegraphics[height=0.7\textwidth, width=1.05\textwidth]{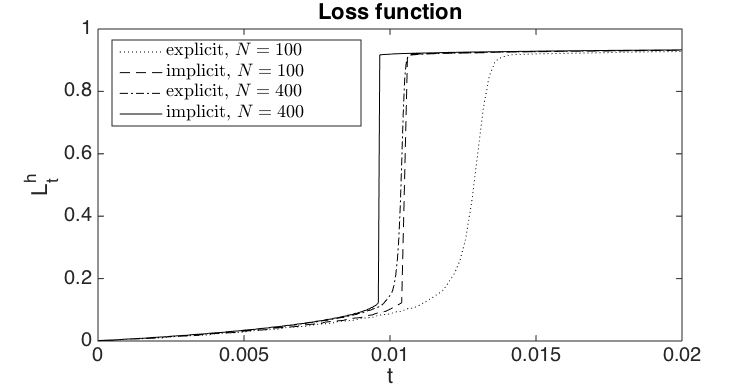}
        \caption{Jump for $\alpha=1.5$, explicit vs. implicit.}
        \label{fig:loss}
    \end{subfigure}%
    \hfill
    \begin{subfigure}[t]{0.49\textwidth}
        \hspace{-0.5cm}
        \includegraphics[height=0.7\textwidth, width=1.05\textwidth]{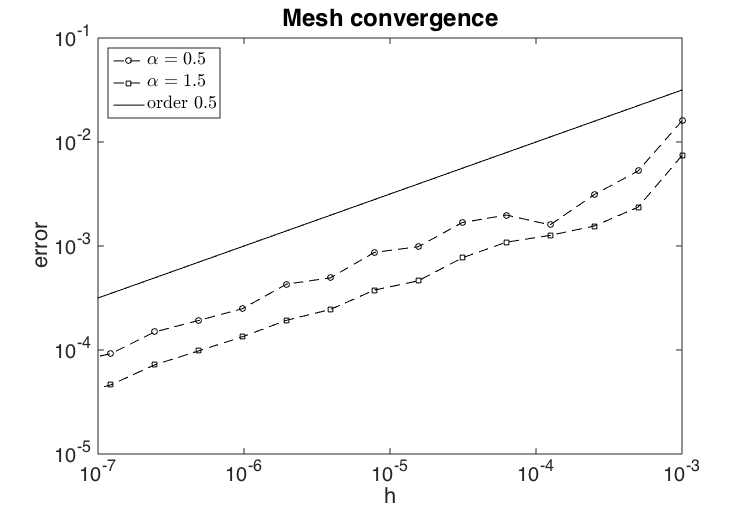}
        
         \caption{Implicit Donsker, $\alpha=0.5$ and $\alpha=1.5$.}      
       \label{fig:alpha}
    \end{subfigure}
     \caption{Convergence of $L^h = \underline{\Lambda}^h/\alpha$ in the Donsker scheme.}
    \end{figure}

Let us now turn to the convergence of $L^h = \Ldonskermin^h/\alpha$  as the step size $h$ goes to 0 in the Donsker scheme.
Figure \ref{fig:alpha} shows the error estimator $2 (L_T^h - L_T^{2 h})$ for decreasing $h$,
for $\alpha = 0.5$ and $\alpha = 1.5$, and otherwise the same parameters as earlier.\footnote{Assuming the error to be 
$L_T - L_T^{h} \approx c h^{1/2}$, we find more precisely $L_T^h - L_T^{2 h} \approx (\sqrt{2}-1) (L_T - L_T^{h})$.
} 
In both cases the order of convergence appears to be $0.5$, irrespective of the jump that occurs for $\alpha=1.5$. 

We now analyse the same example in the jump regime ($\alpha=1.5$), 
for the implicit and explicit Euler-type time-stepping scheme (i.e.\ without Donsker approximation), whose precise difference is explained in Remark \ref{rem:implicit_explicit}.
In the first two experiments, we fix a seed to generate $n=100\, 000$ sample paths 
for the particle method detailed in Section \ref{sec:particle}, and vary the number $N$ of  time points. The relatively small number of particles is chosen to keep the computational time similar to the Donsker scheme.

In Figure \ref{fig:loss_timestepping} we observe the same phenomenon as in Figure \ref{fig:loss} for the Donsker scheme 
when comparing the resolution of the jump between explicit and implicit schemes.
 Indeed, the explicit scheme again smoothes out the jump and takes more time to converge, 
in a way already seen in \cite{kaushansky2020convergence} (see in particular Figure 3 there). More precisely, the implicit scheme with $N=100$ behaves again similarly to the explicit scheme with $N=400$ (compare Figure \ref{fig:loss}).

\begin{figure}[t!]
    \centering
    \begin{subfigure}[t]{0.49\textwidth}
            \hspace{-0.2cm}
        \includegraphics[height=0.7\textwidth, width=1.05\textwidth]{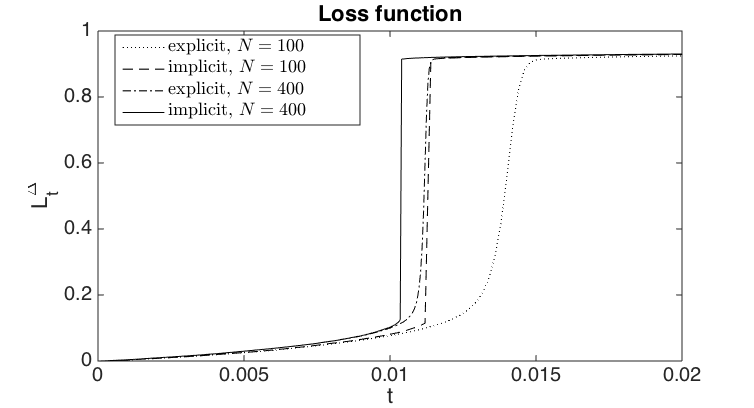}
        \caption{Jump for $\alpha=1.5$, explicit vs. implicit.}
        \label{fig:loss_timestepping}
    \end{subfigure}%
    \hfill
%
       \begin{subfigure}[t]{0.49\textwidth}
        \centering
            \hspace{-0.5cm}
        \includegraphics[height=0.7\textwidth, width=1.05\textwidth]{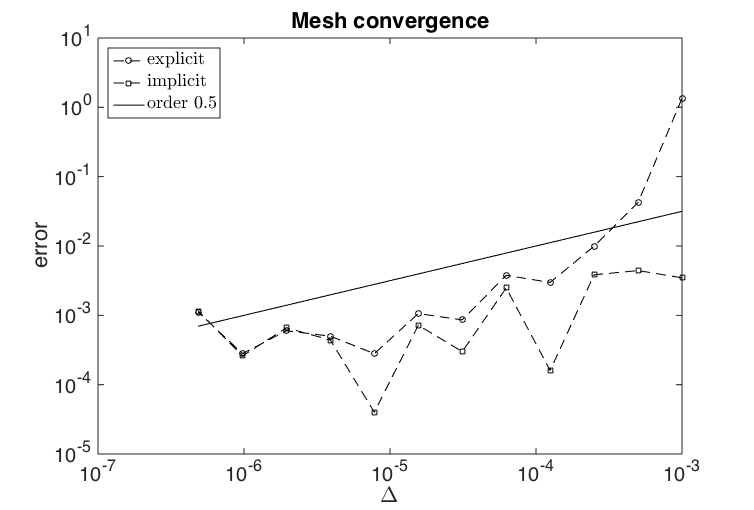}
        \caption{Mesh convergence, explicit vs. implicit.}
        \label{fig:mesh_timestepping}
   \end{subfigure}
    \caption{Convergence of $L^\Delta = \underline{\Lambda}^\Delta/\alpha$ in the time-stepping schemes for $\alpha=1.5$.}
    \end{figure}

Concerning the convergence of $L_T^\Delta = \Ldonskermin_T^\Delta/\alpha$ (for $\Delta=1/N)$, both the implicit and explicit time-stepping scheme show more irregular behaviour than the implicit Donsker approximation as illustrated in Figure \ref{fig:mesh_timestepping}. This is likely to be a consequence of the Monte Carlo error which is quite high due to the relatively small sample size.
Figure \ref{fig:jumps} quantifies this further by showing the error estimator $4 |t^{2\Delta}_*-t_*^{\Delta}|$ for the jump times\footnote{The factor 4 accounts for extrapolation of the timestepping error and additional Monte Carlo error.},
where $t_*^\Delta$ denotes the jump time for
mesh size $\Delta$, identified by $t_*^\Delta = \Delta \; \text{argmax}_{0<k\le N}\{\Ldonskermin^{\Delta}_{t_{k}} - \Ldonskermin^{\Delta}_{t_{k-1}}\}$.
Also shown is the error estimator $4 |J^{2\Delta}- J^{\Delta}|$ for the jump size, where
$J^{\Delta} = L^\Delta_{t*} -  L^\Delta_{t*-\Delta}$.
Here, we choose between $N=2$ and approximately 2000 timesteps, and $2000 N$ samples, to reduce the Monte Carlo error together with the time stepping error.

\begin{figure}[t!]
    \centering
        \includegraphics[height=0.4\textwidth, width=0.6\textwidth]{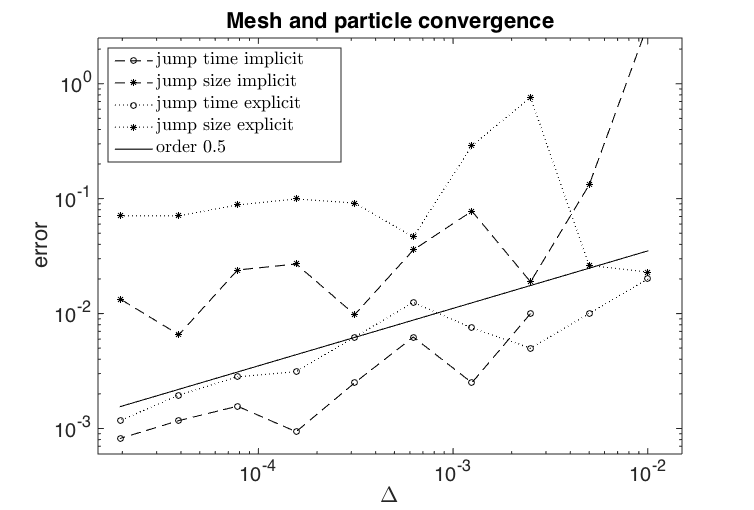}
     \caption{Convergence of jump time and jump size, explicit and implicit time stepping.}
      \label{fig:jumps}
    \end{figure}

The jump times appear to converge with order 1/2 for both the explicit and implicit scheme, where the implicit scheme is slightly more accurate by a constant factor.\footnote{The missing data points in the implicit case are explained by 0 change of the jump time for coarse time steps and resulting undefined values in the log-log plot.} This is consistent with the earlier observation in Figure \ref{fig:loss_timestepping}.
The simple estimate $J^{\Delta}$ of the jump size does not converge to the true jump size of around 0.78 for the explicit scheme, but fluctuates around 0.27. This is due to the fact that we only consider changes over a single time step.
In contrast, the jump size in the implicit scheme appears to converge with order 1/2, albeit with relatively high variance.



One key advantage of the Donsker scheme is the avoidance of Monte Carlo sampling, which explains its
outperformance over the time-stepping algorithm in terms of computational complexity.
We shall therefore focus in the following examples solely on the implicit Donsker approximation, for which 
we investigate two further cases of intermediate regularity, inspired by \cite{kaushansky2020convergence}.

First, we consider the jump regime with $\alpha=1.5$ and  vary $k$ in the $\Gamma(k,1/3)$ initial distribution of $X_{0-}$ to $k=3/2$ and $5/4$, such that the density is only H\"older 1/2 and 1/4, respectively. 
As seen from Figure \ref{fig:k}, the empirical convergence order is still 0.5 in all cases, even though for small $N$ the lower regularity of the initial density is noticeable.

\begin{figure}[t!]
    \centering
    \begin{subfigure}[t]{0.49\textwidth}
            \hspace{-0.2cm}
        \includegraphics[height=0.7\textwidth, width=1.05\textwidth]{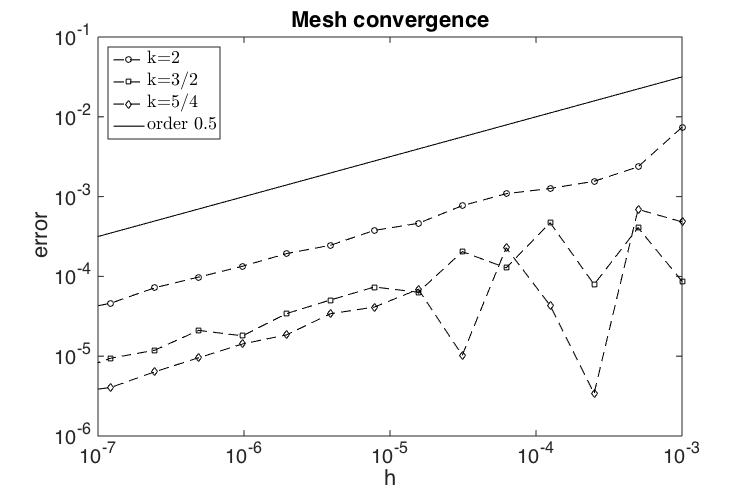}
        \caption{Different $\Gamma(k,1/3)$ initial values, alpha=1.5.}
        \label{fig:k}
    \end{subfigure}%
    \hfill
    \begin{subfigure}[t]{0.49\textwidth}
        \hspace{-0.7cm}
        \includegraphics[height=0.7\textwidth, width=1.05\textwidth]{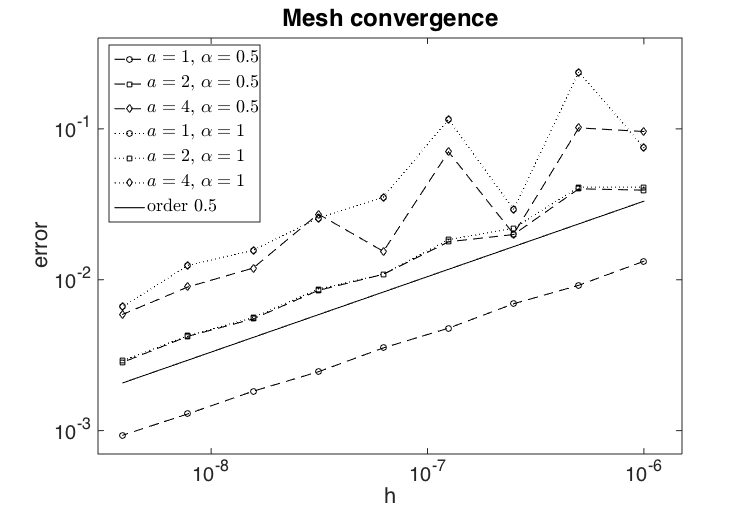}
         \caption{Initial densities as in \eqref{dens_a_alph}.}      
       \label{fig:a}
    \end{subfigure}
     \caption{Mesh convergence, implicit Donsker.}
    \end{figure}

Second, we consider the initial density
\begin{equation}
\label{dens_a_alph}
	V_{0-}(x) = \left\{
	\begin{tabular}{ll}
		$\frac{1}{\alpha} - c x^{a}$, & \quad $0 \le x \le A$, \\
		$0$, & \quad $x > A$,
	\end{tabular} \right.
\end{equation}
for $\alpha>0$ and $a>0$, which we vary in the tests, and where $A>0$ is determined by $\int_0^\infty V_{0-}(x)\,\mathrm{d}x=1$ for given $c>0$, the latter being sufficiently small. Moreover, we let $T=10^{-4}$ be small enough to precede a possible discontinuity.
Here, the convergence order of the explicit time-stepping scheme in \cite{kaushansky2020convergence} is $1/(2(a+1))$ (see Theorem 1.5 and Table 1 there).
Remarkably, the asymptotic order of our scheme appears to be 0.5 irrespective of $a$, see Figure \ref{fig:a}.

%
%
%

\bibliographystyle{plain}
\bibliography{approximation_schemes_arxiv}
\end{document}